\documentclass[12pt,leqno,twoside]{article}
\usepackage{amssymb}
\usepackage{amsmath}
\usepackage{amsthm}
\usepackage{t1enc}
\usepackage[cp1250]{inputenc}
\usepackage{a4,indentfirst,latexsym}
\usepackage{graphics}
\usepackage{mathrsfs}
\usepackage[normalem]{ulem}
\usepackage{cite,enumitem,graphicx}
\usepackage[colorlinks=true,urlcolor=black,
citecolor=black,linkcolor=black,linktocpage,pdfpagelabels,
bookmarksnumbered,bookmarksopen]{hyperref}
\usepackage[colorinlistoftodos]{todonotes}

\input xy
\xyoption{all}

\newtheorem{Th}{Theorem}[section]
\newtheorem{Prop}[Th]{Proposition}
\newtheorem{Lem}[Th]{Lemma}

\newtheorem{Rem}[Th]{Remark}

\newcommand{\wt}{\widetilde}

   \newcommand{\vp}{\varphi}
   
   \newcommand{\eps}{\varepsilon}

   \def\div{\mathop{\mathrm{div}\,}}
   \def\essinf{\mathop{\mathrm{ess\, inf}\,}}

   \def\id{\mathrm{id}}

   \def\Z{\mathbb{Z}}

   \def\N{\mathbb{N}}
   \def\R{\mathbb{R}}

   \def\curl{\mathrm{curl}}
   \def\dim{\mathrm{dim}}

   \def\V{\mathcal{V}}
   
   \def\J{\mathcal{J}}

   \def\W{\mathcal{W}}
   \def\D{\mathcal{D}}

   \DeclareMathOperator*{\esssup}{sup}

\newcommand{\cB}{{\mathcal B}}
\newcommand{\cC}{{\mathcal C}}
\newcommand{\cD}{{\mathcal D}}
\newcommand{\cE}{{\mathcal E}}

\newcommand{\cH}{{\mathcal H}}

\newcommand{\cJ}{{\mathcal J}}

\newcommand{\cL}{{\mathcal L}}
\newcommand{\cM}{{\mathcal M}}
\newcommand{\cN}{{\mathcal N}}
\newcommand{\cO}{{\mathcal O}}
\newcommand{\cP}{{\mathcal P}}

\newcommand{\cT}{{\mathcal T}}

\newcommand{\cV}{{\mathcal V}}
\newcommand{\cW}{{\mathcal W}}

\newcommand{\fJ}{J}

\renewcommand{\dim}{{\rm dim}\,}

\newcommand{\al}{\alpha}
\newcommand{\be}{\beta}
\newcommand{\ga}{\gamma}

\newcommand{\la}{\lambda}
\newcommand{\si}{\sigma}
\newcommand{\om}{\omega}

\newcommand{\Ga}{\Gamma}
\newcommand{\Om}{\Omega}

\def\curlop{\nabla\times}
\newcommand{\weakto}{\rightharpoonup}
\newcommand{\pa}{\partial}
\def\id{\mathrm{id}}

\newcommand{\tX}{\widetilde{X}}
\newcommand{\tu}{\widetilde{u}}
\newcommand{\tv}{\widetilde{v}}
\newcommand{\tcV}{\widetilde{\cV}}

\newcommand{\cTto}{\stackrel{\cT}{\longrightarrow}}

\numberwithin{equation}{section}

\begin{document}
\title{Nonlinear time-harmonic Maxwell equations in domains}
\author{Thomas Bartsch \and Jaros\l aw Mederski
\footnote{The author was partially supported by the National Science Centre, Poland (Grant No. 2014/15/D/ST1/03638).}
}

\date{}
\maketitle

\begin{center}
  {\it Dedicated to Paul Rabinowitz.}
\end{center}

\begin{abstract}
  The search for time-harmonic solutions of nonlinear Maxwell equations in the absence of charges and currents leads to the elliptic equation
  \[
    \curlop\left(\mu(x)^{-1}\curlop u\right) - \om^2\eps(x)u = f(x,u)
  \]
  for the field $u:\Om\to\R^3$ in a domain $\Om\subset\R^3$. Here $\eps(x) \in \R^{3\times3}$ is the (linear) permittivity tensor of the material, and $\mu(x) \in \R^{3\times3}$ denotes the magnetic permeability tensor. The nonlinearity $f:\Om\times\R^3\to\R^3$ comes from the nonlinear polarization. If $f=\nabla_uF$ is a gradient then this equation has a variational structure. The goal of this paper is to give an introduction to the problem and the variational approach, and to survey recent results on ground and bound state solutions. It also contains refinements of known results and some new results.
\end{abstract}



\section{Introduction}\label{sec:intro}

The propagation of electromagnetic waves is described by the Maxwell equations for the electric field $\cE$, the electric displacement field $\cD$, the magnetic field $\cH$, and the magnetic induction $\cB$. These are time-dependent vector fields in a domain $\Om\subset\R^3$. Given the current intensity $\cJ$ and the scalar charge density $\rho$, the Maxwell equations in differential form are as follows:
\begin{equation}\label{eq:Maxwell}
  \left\{
  \begin{aligned}
    &\pa_t \cB + \curlop \cE=0 &&\quad\hbox{(Faraday's Law)}\\
    &\curlop \cH = \cJ+ \pa_t \cD &&\quad\hbox{(Ampere's Law)}\\
    &\div(\cD)=\rho &&\quad\hbox{(Gauss' Electric Law)}\\
    &\div(\cB)=0 &&\quad\hbox{(Gauss' Magnetic Law)}.
  \end{aligned}
  \right.
\end{equation}
These fields are related by constitutive equations determined by the material. The relation between the electric displacement field and the electric field is given by $\cD = \eps\cE +\cP_{NL}(x,\cE)$ where $\eps = \eps(x) \in \R^{3\times3}$ is the (linear) permittivity tensor of the material, and $\cP_{NL}$ is the nonlinear part of the polarization. The relation between magnetic field and magnetic induction is $\cB=\mu\cH-\cM$ where $\mu = \mu(x) \in \R^{3\times3}$ denotes the magnetic permeability tensor and $\cM$ the magnetization of the material. The tensors $\eps, \mu$ are symmetric and positive definite. If the material is isotropic then they are scalar, and if the medium is homogeneous they are constant. In a linear medium one has $\cP_{NL}=0$ leading to the linear Maxwell equations.

Suppose there are no currents, charges nor magnetization, i.e.\ $\cJ=0$, $\rho=0$, $\cM=0$. Then multiplying Faraday's law with $\mu^{-1}$, taking the curl and using the constitutive relations and Ampere's law leads to the nonlinear wave equation
\begin{equation}\label{eq:wave}
  \eps(x)\pa_t^2\cE + \pa_t^2\cP_{NL}(x,\cE) + \curlop (\mu(x)^{-1}\curlop \cE) = 0
\end{equation}
for the electric field $\cE$. Solving this one obtains $\cD = \eps\cE +\cP_{NL}(x,\cE)$ by the constitutive relation and $\cB$ by time integrating Faraday's law. Finally $\cH = \mu^{-1}\cB$ is also determined by the constitutive relation. The fields $\cB$ and $\cD$ will automatically be divergence free provided they are divergence free at time $0$.

The field $\cE$ is said to be time-harmonic (also monochromatic) with frequency $\om>0$ if
\[
  \cE(x,t) = E_1(x)\cos(\om t) + E_2(x)\sin(\om t)\quad\text{for $x\in\Om$ and $t\in\R$.}
\]
The intensity of a time-harmonic field $\cE$ is defined as the time-average
\[
  \frac{1}{T}\int_{0}^{T}|\cE(x,t)|^2dt = |E_1(x)|^2+|E_2(x)|^2
\]
of $|\cE(x,t)|^2$; here $T=2\pi/\om$. Now suppose that the nonlinear polarization is of the form
\[
  \cP_{NL}(x,\cE) = \chi\big(x,|E_1(x)|^2+|E_2(x)|^2\big)\cE
\]
i.e.\ the scalar susceptibility $\chi$ depends only on the intensity of $\cE$. Then \eqref{eq:wave} reduces to the system
\begin{equation}\label{eq:system}
  \left\{
  \begin{aligned}
    \curlop(\mu(x)^{-1}\curlop E_1) - V(x) E_1 &= \chi\big(x,|E_1(x)|^2+|E_2(x)|^2\big)E_1 &&\quad\text{in $\Om$,}\\
    \curlop(\mu(x)^{-1}\curlop E_2) - V(x) E_2 &= \chi\big(x,|E_1(x)|^2+|E_2(x)|^2\big)E_2 &&\quad\text{in $\Om$,}
  \end{aligned}
  \right.
\end{equation}
where  $V(x) = \om^2\eps(x) \in \R^{3\times3}$. Looking for semitrivial solutions where one of $E_1,E_2$ is trivial or where $E_1=E_2$ one is lead to the equation
\begin{equation}\label{eq:main}
  \curlop\left(\mu(x)^{-1}\curlop u\right) - V(x)u = f(x,u) := \chi\big(x,|u|^2\big)u \qquad\text{in $\Om$.}
\end{equation}
Observe that the nonlinearity is a gradient: $f(x,u) = \nabla_u F(x,u)$ with $F(x,u) = \frac12 \psi\big(x,|u|^2\big)$ where $\psi(x,s) = \int_{0}^{s}\chi(x,r)dr$. Let us mention already at this point a major difficulty when dealing with this equation. If $u=\nabla \phi$ is a gradient then $\curlop u=0$, hence the differential operator in \eqref{eq:main} has an infinite-dimensional kernel. This feature is of course already present in the linear Maxwell equations. In order to get around this one uses the Helmholtz decomposition $u=v+w$ with a divergence-free field $v$ and a curl-free field $w$. At first sight this does not seem to be very helpful in the nonlinear setting. We shall see that however that a nonlinear variation of this idea does help.

Probably the most common type of nonlinearity in the physics and engineering literature is the Kerr nonlinearity
\begin{equation}\label{ex:Kerr}
  f(x,u) = \chi^{(3)}(x)|u|^2u.
\end{equation}
Other examples for $f$ that appear in applications are nonlinearities with saturation like
\begin{equation}\label{ex:sat}
  f(x,u) = \chi^{(3)}(x)\frac{|u|^2}{1+|u|^2}u,
\end{equation}
or cubic-quintic nonlinearities like
\begin{equation}\label{ex:cub-quint}
  f(x,u) = \chi^{(3)}(x)|u|^2u - \chi^{(5)}(x)|u|^4u.
\end{equation}
We refer the reader to \cite{Nie,Stuart:1991,Stuart:1993} for these and further examples.

When $\Om$ has a boundary then boundary conditions depend of course on the material characteristics of the complement $\R^3\setminus\overline\Om$. In this survey we shall only consider the case of $\Om$ being surrounded by a perfectly conducting medium which leads to the so-called metallic boundary condition
\begin{equation}\label{eq:bc}
  \nu\times u = 0\qquad\text{on $\pa\Om$}
\end{equation}
where $\nu:\pa\Om\to\R^3$ is the exterior normal.

Solutions of \eqref{eq:main} are critical points of the functional
\begin{equation}\label{eq:action}
  J(u) = \frac12\int_\Om\langle\mu(x)^{-1}\curlop u,\curlop u\rangle\, dx
            - \frac12\int_\Om \langle V(x)u, u\rangle\, dx - \int_\Om F(x,u)\,dx
\end{equation}
defined on an appropriate subspace $X$ of $H_0(\curl;\Om)$ such that $F(x,u)$ and $\langle V(x)u, u\rangle$ are integrable. The precise definition of the domain of $\fJ$ will be given in Section~\ref{sec:var}. The above mentioned difficulty that the curl operator $\curlop$ has an infinite-dimensional kernel is of course also present in the variational approach. One of the consequences is that the functional is strongly indefinite, i.e.\ Morse indices of critical points will be infinite. Another consequence is that the Palais-Smale condition does not hold. And a third difficulty is that the derivative $J':X\to X^*$ is not weak-to-weak$^*$ continuous even when the growth of $F$ is subcritical. Thus even if $\fJ$ has a linking geometry in the spirit of Benci and Rabinowitz \cite{BenciRabinowitz}, the problem cannot be treated by standard variational methods for strongly indefinite functionals as in \cite{BenciRabinowitz,BartschDing,DingBook,MederskiSystem}.

In the literature there are only few results about nonlinear equations like \eqref{eq:main} involving the curl-curl operator. If $\Om=\R^3$ then Benci and Fortunato \cite{BenFor} proposed, within a unified field theory for classical electrodynamics, the equation
\begin{equation}\label{eq:BF}
  \curlop\curlop A = W'(|A|)A
\end{equation}
for the gauge potential $A$ related to the magnetic field $H=\curlop A$. Azzollini et al.\ \cite{BenForAzzAprile} and D'Aprile and Siciliano \cite{DAprileSiciliano} used the cylindrically symmetry of the domain $\R^3$ and of \eqref{eq:BF} in order to find special types of symmetric solutions. Cylindrically symmetric media have also been considered in the work of Stuart and Zhou \cite{Stuart:1991}--\cite{StuartZhou10} on transverse electric and transverse magnetic solutions of \eqref{eq:Maxwell}. The search for these solutions reduces to a one-dimensional variational problem or an ODE, which simplifies the problem considerably. The methods from \cite{Stuart:1991}--\cite{StuartZhou10}
seem to be insufficient to study our problem \eqref{eq:main}. Only very recently new variational methods have been developed that yield critical points of $\fJ$, hence solutions of \eqref{eq:main}. In this survey we present the basic ideas and some of these recent results.

Finally we would like to mention that linear time-harmonic Maxwell equations have been extensively studied by means of numerical and analytical methods, on bounded and unbounded (exterior) domains; see e.g.\ \cite{Ball2012,BuffaAmmariNed,Leis68,Picard01,KirschHettlich,Monk,Doerfler} and the references therein.

The paper is organized as follows. In Section~\ref{sec:var} we discuss the variational approach to \eqref{eq:main} for bounded domains. Then in Section~\ref{sec:sym} we present a theorem in a simpler symmetric setting where standard methods from critical point theory can be applied, which can be found in the seminal paper \cite{Ambrosetti-Rabinowitz:1973} by Ambrosetti and Rabinowitz and the book \cite{Rabinowitz:1986} by Rabinowitz. In order to treat the full strongly indefinite functional we recall some critical point theory developed for \eqref{eq:main}. Then in Section~\ref{sec:res-bounded} we discuss results from \cite{BartschMederski1, BartschMederski2, Qin-Tang:2016, Tang-Qin:2016} about \eqref{eq:main} on bounded domains. The case $\Om=\R^3$ will be discussed in Section~\ref{sec:R^3}. Here we present results from \cite{BDPR:2016, BenForAzzAprile, DAprileSiciliano, HirschReichel, MederskiENZ}. Finally in Section~\ref{sec:problems} we list some open problems.

\section{Variational approach for bounded domains}\label{sec:var}

Throughout the paper we assume that $\Om\subset\R^3$ is bounded or $\Om=\R^3$. In this section we discuss the bounded domain case where we require:
\begin{itemize}
\item[(L1)] $\Om\subset\R^3$ is a bounded domain with Lipschitz boundary. The tensor fields $\mu,V\in L^{\infty}(\Om,\R^{3\times 3})$ satisfy: $\mu(x),V(x)$ are symmetric and positive definite uniformly for $x\in\Om$.
\end{itemize}
Now we define the basic spaces in which we look for solutions of \eqref{eq:main}. The space $L^2_V(\Om,\R^3)$ consists of all measurable vector fields $u:\Om\to\R^3$ such that that $\langle V(x)u,u\rangle\in L^1(\Om)$. This is a Hilbert space with scalar product
\[
  \langle u_1,u_2\rangle_V = \int_\Om \langle V(x)u_1,u_2\rangle dx
\]
and associated norm $|\,.\,|_V$. Clearly, (L1) implies that $L^2_V(\Om,\R^3) = L^2(\Om,\R^3)$, with equivalent norms. The Hilbert space $H_{V,0}(\curl;\Om)$ is by definition the completion of $\cC^{\infty}_0(\Om,\R^3)$ with respect to the norm
\[
  \|u\|_{H_V(\curl;\Om)} := \left(|\curlop u|^2_2+|u|_V^2\right)^{1/2}.
\]
Here $\curlop u$ has to be understood in the distributional sense, and $|\cdot|_q$ denotes the $L^q$-norm. Setting
\[
  \langle \curlop u_1,\curlop u_2\rangle_{\mu^{-1}} = \int_\Om \langle \mu(x)^{-1}\curlop u_1,\curlop u_2\rangle dx
\]
with associated semi-norm $|\,.\,|_{\mu^{-1}}$ assumption (L1) implies that $\|u\|_{H_V(\curl;\Om)}$ is equivalent to the norm
\[
  \|u\|_{\mu,V} := \left(|\curlop u|_{\mu^{-1}}^2 + |u|_V^2\right)^{1/2}.
\]
Also by (L1) the space $H_{V,0}(\curl;\Om)$ is equivalent to the space denoted $H_0(\curl;\Om)$ in the literature. Observe that elements of $H_{V,0}(\curl;\Om)$ need not be zero on the boundary. In fact, for $u\in H^1_0(\Om)$ we claim that $\nabla u\in H_{V,0}(\curl;\Om)$. There exists $\phi_n\in \cC^{\infty}_0(\Om)$ converging towards $u$ in $H^1(\Om)$ and such that $\nabla\phi_n$ converges towards $\nabla u$ in $L^2_V(\Om,\R^3)$. Then $\nabla\phi_n$ converges towards $\nabla u$ in $H_V(\curl;\Om)$ because the curl of gradient fields is $0$. Vector fields $u\in H_{V,0}(\curl;\Om)$ satisfy the boundary condition $\nu\times u=0$ on $\pa\Om$ in the weak sense.

Next we discuss the Helmholtz decomposition. The space
\[
  \cV_0 = \left\{v\in H_{V,0}(\curl;\Om): \int_\Om\langle V(x)v,\phi\rangle\,dx=0
          \text{ for every $\phi\in \cC^\infty_0(\Om,\R^3)$ with $\curlop\phi=0$} \right\}
\]
consists of vector fields $v\in H_{V,0}(\curl;\Om)$ such that $V(x)v$ is divergence-free in the distributional sense. The space
\[
  \cW_0
   = \left\{w\in H_{V,0}(\curl;\Om):\int_\Om\langle w,\curlop\phi\rangle = 0
     \text{ for all }\phi\in\cC^\infty_0(\Om,\R^3)\right\}
\]
consists of curl-free vector fields in $H_{V,0}(\curl;\Om)$, in the distributional sense. Since for every $\phi\in\cC^\infty_0(\Om;\R^3)$ the linear map
\[
  u \mapsto \int_\Om \langle u,\curlop\varphi\rangle dx
\]
is continuous on $H_{V,0}(\curl;\Om)$, the space $\cW_0$ is a closed complement of $\cV_0$ in $H_{V,0}(\curl;\Om)$, hence there is a Helmholtz type decomposition
\begin{equation}\label{eq:Helmholtz}
  H_{V,0}(\curl;\Om) = \cV_0\oplus\cW_0.
\end{equation}
Therefore any $u\in H_{V,0}(\curl;\Om)$ can be decomposed as $u=v+w$ with $v\in\cV_0$ and $w\in\cW_0$, where $V(x)v$ is divergence-free and $w$ is curl-free.

\begin{Lem}\label{lem:evp}
  The curl-curl source eigenvalue problem
  \begin{equation}\label{EgEigenvalue}
    \left\{
    \begin{aligned}
      &\curlop(\mu(x)^{-1}\curlop v) = \la V(x)v &&\quad \hbox{in } \Om,\\
      &\nu\times v =0 &&\quad \hbox{on } \pa\Om\\
      &v\in\cV_0
    \end{aligned}
    \right.
  \end{equation}
  has a discrete sequence $0<\la_1<\la_2<\la_3<\ldots$ of (anisotropic) Maxwell eigenvalues with eigenspaces of finite multiplicity. The quadratic form $Q:\cV_0\to\R$ defined by
  \begin{equation}\label{eq:DefofQ}
    Q(v) := \int_\Om \left(\langle\mu(x)^{-1}\curlop v,\curlop v\rangle - \langle V(x)v,v\rangle\right)\,dx,
  \end{equation}
  is positive definite on the sum $\cV^+\subset\cV_0$ of the eigenspaces associated to the eigenvalues $\la_k>1$ and it is negative semi-definite on the sum $\tcV\subset\cV_0$ of the eigenspaces associated to the eigenvalues $\la_k\le1$.
\end{Lem}

\begin{proof}
  By \cite[Theorem~4.7]{Bauer-Pauly-Schomburg:2016} the space $\cV_0$ embeds compactly into $L^2_V(\Om,\R^3)$. The lemma follows immediately.
\end{proof}

Under slightly more rigorous regularity conditions $\cV_0$ embeds even compactly into $L^p_V(\Om,\R^3)$ for $2\le p<6$, see \cite[Proposition~3.1]{BartschMederski2}, but the space $H_{V,0}(\curl;\Om)$ does not embed into $L^p(\Om,\R^3)$ for $p>2$. Therefore the functional $J$ from \eqref{eq:action} will in general only be defined on a smaller subspace $X\subset H_{V,0}(\curl;\Om)$ that depends on the nonlinearity $F$. Our basic conditions on $F$ are as follows.
\begin{itemize}
\item[(F1)] $F:\Om\times\R^3\to\R$ is differentiable with respect to $u\in\R^3$, such that $f=\nabla_uF:\Om\times\R^3\to\R^3$ is a Carath\'eodory function (i.e.\ measurable in $x\in\Om$, continuous in $u\in\R^3$ for a.e.\ $x\in\Om$). Moreover, $F(x,0)=0$ for a.e.\ $x\in\Om$.
\item[(F2)] $|f(x,u)|=o(|u|)$ as $u\to0$ uniformly in $x\in\Om$.
\item[(F3)] There exist $2<p<6$ and $c>0$ such that
    \[
      |f(x,u)|\le c(1+|u|^{p-1})\qquad\text{for all } x\in\Om, u\in\R^3.
    \]
\end{itemize}

Then $J(u)$ is defined for $u \in X := H_{V,0}(\curl;\Om)\cap L^p(\Om,\R^3)$. Recall that $p=6$ is the critical Sobolev exponent in dimension $3$ so that we require subcritical growth in (F3). Curl-curl equations with critical growth have been treated only very recently (see \cite{MederskiMaxwellCritical, Zeng:2016} and will not be treated here.

Observe that the Helmholtz decomposition \eqref{eq:Helmholtz} induces a Helmholtz decomposition of $X=\cV\oplus\cW$ where $\cV:=\cV_0\cap X$ and $\cW:=\cW_0\cap X$. Now we can formulate the variational nature of \eqref{eq:main}; see \cite{BartschMederski1,BartschMederski2}.

\begin{Prop}\label{PropSolutE}
The functional $J:X = H_{V,0}(\curl;\Om)\cap L^p(\Om,\R^3)\to\R$ given by \eqref{eq:action} is of class $\cC^1$. Moreover $u\in X$ is a critical point of $J$ if and only if $u$ is a (weak) solution of \eqref{eq:main}.
\end{Prop}

In order to illustrate the difficulties in dealing with $J$ consider the model case $F(x,u)=\frac1p|u|^p$ so that
\[
  J(u) = \frac12\int_\Om\langle\mu(x)^{-1}\curlop u,\curlop u\rangle\,dx
            - \frac12\int_\Om \langle V(x)u, u\rangle\,dx - \frac1p\int_\Om |u|^p\,dx.
\]
Then $J|_\cV$ has mountain pass geometry and $J|_\cW$ is strictly concave with $0$ as global maximum. All nontrivial critical points have an infinite Morse index because for $\psi\in\cW$ there holds
\[
J''(u)[\psi,\psi] = - \int_\Om \langle V(x)\psi, \psi\rangle\, dx - (p-1)\int_\Om |u|^{p-2}|\psi|^2\,dx \le 0.
\]
An additional difficulty is that $J':X\to X^*$ is not sequentially weak-to-weak$^*$ continuous. Therefore the critical point theory for strongly indefinite functionals from \cite{BartschDing,BenciRabinowitz,DingBook,MederskiSystem} does not apply.

\section{Symmetry}\label{sec:sym}

Let us first consider the fully radially symmetric case where $\Om\subset\R^3$ may be a ball, an annulus, the exterior of a ball, or all of $\R^3$. For simplicity we only deal with the equation
\begin{equation}\label{eq:main-rad}
  \curlop\curlop u + V(|x|)u = \Ga(|x|)|u|^{p-2}u  \quad \mbox{ in } \Om,
\end{equation}
and assume that $V,\Ga:I\to\R$ with $I=\{|x|:x\in\Om\}\subset[0,\infty)$. Then the following holds.

\begin{Th}\label{thm:radial}
Let $p>2$ and suppose that $V,\Ga\in L^\infty_{loc}(I)$ and $0\le V\Ga^{-1}\in L_{loc}^{\frac{p-1}{p-2}}(\Om)$. Let $u\in L^{p-1}_{loc}(\R^3)$ be a distributional solution of \eqref{eq:main-rad} such that $u(x)=M^T u(Mx)$ for a.a. $x\in \Om$ and all $M\in O(3)$. Then $\nabla\times u=0$, $V(r)\Ga(r)\ge0$ for all $r\in I$, and there exists a measurable function $s:I \to \{-1,1\}$ such that
\begin{equation}\label{eq:sol-rad}
u(x) = s(|x|) \left(\frac{V(x)}{\Ga(x)}\right)^\frac{1}{p-2} \frac{x}{|x|}.
\end{equation}
Conversely, any $u$ as in \eqref{eq:sol-rad} is curl-free and solves \eqref{eq:main-rad}.
\end{Th}

The theorem has been proved in \cite[Theorem~1]{BDPR:2016} in the case of $\Om=\R^3$ but the proof works for any radial domain. Observe that for a curl-free field $u$ the equation \eqref{eq:main-rad} reduces to an algebraic equation. Thus the assumption of full radial symmetry does not lead to interesting solutions. We therefore relax the fully radial symmetry and look for solutions on cylindrically symmetric domains having cylindrical symmetry. These are in fact of great importance due to the phenomenon of birefringence and applications in crystallography \cite{FundPhotonics,StuartZhou03,Nie}. We allow cylindrically symmetric anisotropic materials. More precisely we require that the problem is symmetric with respect to the cylindrical symmetry group
\[
  G = O(2)\times\{1\}
   = \left\{\begin{pmatrix}
              \cos \alpha  & -\sin\alpha & 0 \\
              \sin\alpha & \cos\alpha & 0 \\
              0 & 0 & 1
            \end{pmatrix} : \alpha \in \R\right\}\subset O(3)
\]
in the following sense:
\begin{itemize}
\item[(S)] $\Om$ is invariant with respect to $G$, and $F:\Om\times\R^3\to\R$ is invariant with respect to the action of $G$ on the $x$- and $u$-variables, i.e.\ $F(g_1x,g_2u)=F(x,u)$ for all $x\in\Om$, $u\in\R^3$, $g_1,g_2\in G$. Moreover, $\mu(x)$ and $V(x)$ commute with $G$, and $\mu,V$ are invariant with respect to $G$, i.e.\ $g_2\mu(g_1x)g_2^{-1}=\mu(x)$ for all $x\in\Om$, $g_1,g_2\in G$; similarly for $V$.
\end{itemize}
The invariance of $F$ with respect to $G$ is equivalent to the statement that
\[
  F(x,u)=F\left(\sqrt{x_1^2+x_2^2},x_3,\sqrt{u_1^2+u_2^2},u_3\right) \qquad\text{holds for $x\in\Om$, $u\in\R^3$.}
\]
That the permeability tensor $\mu(x)$ commutes with $G$ is equivalent to
\[
  \mu(x) = \begin{pmatrix}a(x) & 0 & 0\\0 & a(x) & 0\\0 & 0 & b(x)\end{pmatrix},
\]
with $a,b\in L^{\infty}(\Om)$ positive, bounded away from $0$, and invariant with respect to the action of $G$ on $\Om$; similarly for $V(x)$, hence for the permittivity tensor $\eps(x)$.

The first existence theorem for solutions of \eqref{eq:main} deals with superlinear nonlinearities, e.g.\ Kerr type nonlinearities. We assume the following Ambrosetti-Rabinowitz condition.
\begin{itemize}
  \item[(F4)] There exists $\be>2$ and $R>0$ such that $\langle f(x,u),u\rangle \ge \be F(x,u) > 0$ for all $x\in\Om$ and all $u\in\R^3$ with $|u|\ge R$.
\end{itemize}

The following result is due to \cite[Theorem~2.5]{BartschMederski2}.

\begin{Th}\label{thm:sym1-superlin}
  Suppose (L1), (S), (F1)-(F4) hold and suppose that $F$ is even in $u$: $F(x,-u)=F(x,u)$. Then there exist infinitely many solutions $u_n$ of the form
  \begin{equation}\label{eq:sym1}
    u(x)=\al(r,x_3)\begin{pmatrix}-x_2\\x_1\\0\end{pmatrix},\qquad r=\sqrt{x_1^2+x_2^2},
  \end{equation}
  and such that $J(u_n)\to\infty$.
\end{Th}

We give a sketch of the proof.
\begin{proof}
Since $\Om$ is invariant under $G=O(2)\times\{1\}\subset O(3)$ we can define an action of $g\in G$ on $u\in L^2(\Om,\R^3)$ by setting
\begin{equation}\label{eq:G-action}
  (g*u)(x) := g\cdot u(g^{-1}x).
\end{equation}
This action leaves $X = H_{V,0}(\curl;\Om)\cap L^p(\Om,\R^3)$ and the subspaces $\cV,\cW$ invariant. The fixed point set $X^G=\cV^G\oplus\cW^G$ consists of all $G$-equivariant vector fields. In view of \cite[Lemma 6.2]{BartschMederski2}, any $u\in W_0^p(\curl;\Om)^G$ has a unique decomposition $u=u_\tau+u_\rho+u_\zeta$ with summands of the form
\[
  u_\tau(x) = \al(r,x_3)\begin{pmatrix}-x_2\\x_1\\0\end{pmatrix},\;
  u_\rho(x) = \be(r,x_3)\begin{pmatrix}x_1\\x_2\\0\end{pmatrix},\;
  u_\zeta(x) = \ga(r,x_3)\begin{pmatrix}0\\0\\1\end{pmatrix}.
\]
The map
\begin{equation}\label{eq:def-S1}
  S_1:W_0^p(\curl;\Om)^G \to W_0^p(\curl;\Om)^G, \quad S_1(u_\tau+u_\rho+u_\zeta) := u_\tau-u_\rho-u_\zeta
\end{equation}
is a linear isometry. The symmetry condition (S) implies that $J(u_\tau+u_\rho+u_\zeta)=J(-u_\tau+u_\rho+u_\zeta)$. In fact, the equalities $|\curlop u|_{\mu^{-1}}^2=|\curlop u_\tau|_{\mu^{-1}}^2+|\curlop(u_\rho+u_\zeta)|_{\mu^{-1}}^2$ and $|u|_V^2=|u_\tau|_V^2+|u_\rho|_V^2+|u_\zeta|_V^2$ hold (even pointwise), see \cite{BenForAzzAprile}. In addition $F(x,u(x))=F(x,-S_1(u)(x))$ holds by (S), hence $F(x,u(x))=F(x,S_1(u)(x))$ holds if $F$ is even in $u$.

Therefore it is sufficient to find critical points of $J$ constrained to the fixed point set
\begin{equation}\label{eq:S1-fp}
  (X^G)^{S_1} := \{u\in X^G: S_1(u)=u\} = \{u\in X^G: u=u_\tau\}\subset\cV.
\end{equation}

Observe that $u=u_\tau$ implies $\div(V(x)u)=0$ because $V=V(r,x_3)$, hence $(X^G)^{S_1} \subset \cV$. Moreover, the boundary condition $\nu\times u=0$ implies that $u=0$ on the boundary, hence $(X^G)^{S_1} \subset H^1_0(\Om,\R^3)$ embeds compactly into $L^p(\Om,\R^3)$. Consequently $J|(X^G)^{S_1}$ satisfies the Palais-Smale condition. Using (F4) and Lemma~\ref{lem:evp} it is easy to verify the hypotheses of the symmetric mountain pass theorem \cite{Ambrosetti-Rabinowitz:1973,Rabinowitz:1986} or of the fountain theorem \cite[Theorem~2.5]{Bartsch:1993}.
\end{proof}

We can also treat asymptotically linear nonlinearities, in particular nonlinearities with saturation like \eqref{ex:sat}. To state one result in this direction we assume the following.
\begin{itemize}
  \item[(F5)] There exists $V_\infty\in L^{\infty}(\Om,\R^{3\times 3})$ with $V_\infty(x)$ being symmetric and $V_\infty(x)+C_\infty\id$ positive definite uniformly for $x\in\Om$, some $C_\infty\in\R$, such that $f(x,u)=V_\infty(x)[u]+o(|u|)$ as $|u|\to\infty$.
\end{itemize}

By Lemma~\ref{lem:evp} the quadratic forms
\[
  Q_0(v) := \int_\Om \left(\langle\mu(x)^{-1}\curlop v,\curlop v\rangle - \langle V(x)v,v\rangle \right)\,dx,
\]
and
\[
  Q_\infty(v)
   := \int_\Om \left(\langle\mu(x)^{-1}\curlop v,\curlop v\rangle - \langle (V(x)+V_\infty(x))v,v\rangle \right)\,dx,
\]
defined on $(X^G)^{S_1} \subset \cV$ as in \eqref{eq:S1-fp} have a finite index. Using the constrained functional $J|(X^G)^{S_1}$ and standard critical point theory one can prove the following theorem.

\begin{Th}\label{thm:sym1-asy-lin}
  Suppose (L1), (S), (F1)-(F3), (F5) hold and that $F$ is even in $u$. Suppose moreover that the quadratic forms $Q_0,Q_\infty:(X^G)^{S_1}\to\R$ are non-degenerate with indices $i_0,i_\infty\in\N_0$, respectively. Then there exist at least $|i_0-i_\infty|$ many nontrivial pairs of solutions $\pm u_n$ of the form \eqref{eq:sym1}.
\end{Th}

The case $i_0>i_\infty$ follows from \cite[Theorem~12]{Clark:1972}, the case $i_\infty>i_0$ from \cite{Bartolo-etal:1983}. The results from \cite{Bartolo-etal:1983} allow even to consider the case when $Q_0,Q_\infty$ are degenerate.

\begin{Rem}\label{rem:sym1-sat}
  Theorem~\ref{thm:sym1-asy-lin} applies to the nonlinearity with saturation \eqref{ex:sat}:
  \[
    f(x,u) = \chi^{(3)}(x)\frac{|u|^2}{1+|u|^2}u = \chi^{(3)}(x)u + o(|u|) \qquad\text{as $|u|\to\infty$}
  \]
  when $V_\infty(x)=\chi^{(3)}(x)$ is scalar and invariant under $G$. One sets $F(x,u)=\frac12\chi^{(3)}(x)H(|u|^2)$ where $H(t)=\int_0^t\frac{s^2}{1+s^2}\,ds$. Observe that $\chi^{(3)}(x)$ is the cubic susceptibility for $|u|$ small. Theorem~\ref{thm:sym1-asy-lin} may be interpreted as saying that the larger $\chi^{(3)}$ is, the more solutions of the form \eqref{eq:sym2} exist.
\end{Rem}

The evenness of $F$ in $u$ in Theorems~\ref{thm:sym1-superlin} and \ref{thm:sym1-asy-lin} was needed to prove that $J(S_1(u))=J(u)$. Without $F$ being even it follows from (S) that $J(-S_1(u))=J(u)$ which suggests to look for solutions that are fixed by $S_2=-S_1$. These solutions are of the form
\begin{equation}\label{eq:sym2}
  u(x)=\be(r,x_3)\begin{pmatrix}x_1\\x_2\\0\end{pmatrix} + \ga(r,x_3)\begin{pmatrix}0\\0\\1\end{pmatrix}.
\end{equation}
However the space $(X^G)^{S_2}$ of such functions does not embed into $\cV$ nor does it embed compactly into $L^p(\Om,\R^3)$. Therefore this does not lead to a simpler setting than looking for critical points of $J$ in the full space. In order to obtain such solutions we need the critical point theory from \cite{BartschMederski1} which we present in the next section.

The idea to look for solutions of \eqref{eq:main} of the form \eqref{eq:sym1} and to use the action of $S_1$ on cylindrically symmetric vector fields is due to \cite{BenForAzzAprile} for a special class of curl-curl equations on $\R^3$. For this class solutions of the form \eqref{eq:sym2} have been obtained in \cite{DAprileSiciliano}. We shall present these results in Section~\ref{sec:R^3}.

\section{Critical point theory}\label{sec:CriticalTheory}

In order to treat the full functional $J$ from \eqref{eq:action} so far in all papers on the topic more rigorous hypotheses are required for $F$, in particular $F(x,u)$ has to be convex in $u$. This allows the following approach. Decompose $X = X^+\oplus \tX$ so that $J$ has the form
\begin{equation}\label{EqJ}
  J(u) = \frac12\|u^+\|^2-I(u) \quad\text{for $u=u^++\tu \in X^+\oplus \tX$}
\end{equation}
with a convex functional $I$. The space $\tX$ contains in particular the space $\cW$, i.e.\ the infinite-dimensional kernel of the curl operator. Then one can maximize $J(u^++\tu)$ for fixed $u^+$, obtaining a function $\tilde m:X^+\to\tX$. Setting $m(u^+) = u^+ + \tilde m(u^+)$ a critical point $u^+$ of $J\circ m$ corresponds to a critical point $m(u^+)$ of $J$. This is reminiscent of the approach to the linear Maxwell equations when one splits off the curl-free part and solves for the divergence-free part. Of course this approach is also well known for strongly indefinite functionals, but its realization for $J$ requires some new ideas. We present here the critical point theory from \cite{BartschMederski1,BartschMederski2}. This may be useful also for other strongly indefinite problems with an infinite-dimensional kernel of the differential operator, like nonlinear wave equations.

Let $X$ be a reflexive Banach space with norm $\|\cdot\|$ and with a topological direct sum decomposition $X=X^+\oplus\tX$, where $X^+$ is a Hilbert space with a scalar product. For $u\in X$ we denote by $u^+\in X^+$ and $\tu\in\tX$ the corresponding summands so that $u = u^++\tu$. We may assume that $\langle u,u \rangle = \|u\|^2$ for any
$u\in X^+$ and that $\|u\|^2 = \|u^+\|^2+\|\tu \|^2$. The topology $\cT$ on $X$ is defined as the product of the norm topology in $X^+$ and the weak topology in $\tX$. Thus $u_n\cTto u$ is equivalent to $u_n^+\to u^+$ and $\tu_n\weakto\tu$.

Let $J$ be a functional on $X$ of the form \eqref{EqJ}. The set
\begin{equation}\label{eq:ConstraintM}
  \cM := \{u\in X:\, J'(u)|_{\tX}=0\}=\{u\in X:\, I'(u)|_{\tX}=0\},
\end{equation}
obviously contains all critical points of $J$. Suppose the following assumptions hold.
\begin{itemize}
\item[(I1)] $I\in\cC^1(X,\R)$ and $I(u)\ge I(0)=0$ for any $u\in X$.
\item[(I2)] $I$ is $\cT$-sequentially lower semicontinuous:
    $u_n\cTto u\quad\Longrightarrow\quad \liminf I(u_n)\ge I(u)$
\item[(I3)] If $u_n\cTto u$ and $I(u_n)\to I(u)$ then $u_n\to u$.
\item[(I4)] $\|u^+\|+I(u)\to\infty$ as $\|u\|\to\infty$.
\item[(I5)] If $u\in\cM$ then $I(u)<I(u+v)$ for every $v\in \tX\setminus\{0\}$.
\end{itemize}
Clearly (I5) is satisfied for a strictly convex functional $I$. The following proposition has been proved in \cite[Proof of Theorem~4.4]{BartschMederski2}.

\begin{Prop}\label{prop:cM}
  If $I$ satisfies (I1)-(I5) then the functional $J$ from \eqref{EqJ} and $\cM$ from \eqref{eq:ConstraintM} have the following properties.

  a) For each $u^+\in X^+$ there exists a unique $\tu\in \tX$ such that $m(u^+):=u^++\tu\in\cM$. This $m(u^+)$ is the minimizer of $I$ on $u^++\tX$.

  b) $m:X^+\to \cM$  is a homeomorphism with inverse $\cM\ni u\mapsto u^+\in X^+$.

  c) $J\circ m:X^+\to\R$ is $\cC^1$.

  d) $(J\circ m)'(u^+) = J'(m(u^+))|_{X^+}:X^+\to\R$ for every $u^+\in X^+$.

  e) $(u_n^+)_n\subset X^+$ is a Palais-Smale sequence for $J\circ m$ if, and only if, $(m(u_n^+))_n$ is a Palais-Smale sequence for $J$ in $\cM$.

  f) $u^+\in X^+$ is a critical point of $J\circ m$ if, and only if, $m(u^+)$ is a critical point of $J$.

  g) If $J$ is even, then so is $J\circ m$.
\end{Prop}

Observe that $m$ need not be $\cC^1$, and $\cM$ need not be a differentiable manifold, because $I'$ is only required to be continuous. As a consequence of Proposition~\ref{prop:cM} it remains to find critical points of $J\circ m$. This requires additional assumptions in order to apply classical critical point theorems like the mountain pass theorem to $J\circ m$.
\begin{itemize}
\item[(I6)] There exists $r>0$ such that $a:=\inf\limits_{u\in X^+,\|u\|=r} J(u)>0$.
\item[(I7)] There exists $u^+\in X^+$ such that $\sup_{v\in\tX} J(u+v) < a$.
\item[(I8)] $I(t_nu_n)/t_n^2\to\infty$ if $t_n\to\infty$ and $u_n^+\to u^+\ne 0$ as $n\to\infty$.
\end{itemize}
It is not difficult to see that (I8) implies (I7). The functional $J$ is said to satisfy the {\em $(PS)_c^\cT$-condition in} $\cM$ if every $(PS)_c$-sequence $(u_n)_n$ for the unconstrained functional and such that $u_n\in\cM$ has a subsequence which converges in the $\cT$-topology:
\[
u_n \in \cM,\ J'(u_n) \to 0,\ J(u_n) \to c \qquad\Longrightarrow\qquad
u_n \cTto u\in X\ \text{ along a subsequence.}
\]
We shall use this concept below also for other subsets $\cN$ of $X$ instead of $\cM$.

\begin{Th}\label{thm:mpt} Suppose (I1)-(I7) hold and set
  \[
    c_\cM := \inf_{\ga\in\Ga} J(\ga(t))
  \]
  where
  \[
    \Ga = \{\ga\in\cC^0([0,1],\cM):\ga(0)=0,\ \|\ga(1)^+\|>r, \text{ and } J(\ga(1))<a\}.
  \]
  Then the following holds.

  a) $c_{\cM}\ge a>0$ and $J$ has a $(PS)_{c_\cM}$-sequence in $\cM$.

  b) If $J$ satisfies the $(PS)_{c_\cM}^\cT$-condition in $\cM$ then $c_\cM$ is achieved by a critical point of $J$.

  c) If $J$ satisfies (I8), the $(PS)_c^\cT$-condition in $\cM$ for every $c$, and if $J$ is even then it has an unbounded sequence of critical values.
\end{Th}

The proof can essentially be found in \cite[Theorem~4.3]{BartschMederski2}. The only difference is that there (I8) is assumed also for parts a) and b), but an inspection of the proof shows that (I7) is sufficient. Assumption (I8) holds for $I(u)$ growing superquadratically in $u$ as $|u|\to\infty$. This condition implies that $J(tu)\to-\infty$ as $t\to\infty$ for every $u\in X\setminus\tX$. Together with (I6) this yields the typical geometry of the (symmetric) mountain pass theorem for $J\circ m$. For the classical mountain pass theorem (I7) suffices. This is useful for dealing with asymptotically quadratic $I$.

An interesting problem consists in finding ground state solutions as minimizers on a suitable constraint. The Nehari manifold is a natural constraint that has proved to be very useful provided the quadratic part of the functional is positive definite. An extension to indefinite functionals is due to Pankov \cite{Pankov} in the setting of nonlinear Schr\"odinger equations, and independently to \cite{Pistoia-Ramos:2008} in the setting of elliptic systems. For an abstract version see \cite{SzulkinWethHandbook}.

We consider the set
\begin{equation}\label{eq:NehariDef}
\cN := \{u\in X\setminus\tX: J'(u)|_{\R u\oplus \tX}=0\} = \{u\in\cM\setminus\tX: J'(u)[u]=0\} \subset\cM.
\end{equation}
This set is especially useful if for each $u^+\in X^+\setminus\{0\}$ the functional $J$ has a unique critical point $n(u^+)$ on the half space $\R^+u^+ +\tX$, and if moreover $n(u^+)$ is the global maximum of $J$ on the half space $\R^+u^+ +\tX$. Then the map
\begin{equation}\label{eq:def-n}
  n:SX^+=\{u^+\in X^+: \|u^+\|=1\} \to \cN
\end{equation}
is a homeomorphism and the set $\cN$ is a topological manifold, the Nehari-Pankov manifold. It turns out that it is sufficient to find critical points of $J$ by looking for critical points of $J\circ n$. This is the approach from \cite{SzulkinWethHandbook}. In order to realize it we require the following condition on $I$:
\begin{itemize}
  \item[(I9)] $\frac{t^2-1}{2}I'(u)[u] + tI'(u)[v] + I(u) - I(tu+v) < 0$ for every $u\in \cN$, $t\ge 0$, $v\in \tX$ such that $u\neq tu+v$.
\end{itemize}

Since (I9) is a technical condition let us discuss it a bit.

\begin{Rem}\label{rem:I8} \rm
  a) If (I9) holds for functionals $I_1,I_2$ then it also holds for positive linear combinations $\al_1 I_1+\al_2 I_2$, $\al_1,\al_2\ge0$, $\al_1+\al_2>0$. If $I(u)=\frac12\langle Lu,u\rangle$ is a quadratic form with a selfadjoint operator $L$ then $\frac{t^2-1}{2}I'(u)[u] + tI'(u)[v] + I(u) - I(tu+v) = -\frac12\langle Lv,v\rangle$. In this case (I9) holds if $I$ is positive on $\tX$. If $I$ is positive semi-definite then the weak inequality holds in (I9).

  b) In applications $I$ is of the form $I(u)=\int_\Om F(x,u(x))\,dx$ for some class of fields $u:\Om\to\R^N$. Then (I9) is of course a consequence of the corresponding property of $F:\Om\times\R^N\to\R$ where $F'=\pa_uF$:
  \begin{equation}\label{eq:I9forF}\textstyle
     \frac{t^2-1}{2}F'(x,u)[u] + tF'(x,u)[v] + F(x,u) - F(x,tu+v) < 0  \qquad\text{for all $u,v\in\R^N$, all $x\in\Om$.}
  \end{equation}
  We discuss this condition in Remark~\ref{rem:F9} below.
\end{Rem}

We can now describe the Nehari-Pankov manifold as follows; see \cite{BartschMederski2}, in particular Proposition~4.1.

\begin{Prop}\label{prop:nehari}
Suppose (I1)-(I4) and (I6), (I8), (I9) hold. Then for every $u^+ \in SX^+ := \{u\in X^+:\|u\|=1\}$ the functional $J$ constrained to $\R u^++\tX = \{tu^++v:t\in\R,\ v\in\tX\}$ has precisely two critical points with positive energy: $u_1=t_1u+v_1$ and $u_2=t_2u+v_2$ where $t_1>0>t_2$, $v_1,v_2\in\tX$. Moreover, $u_1$ is the unique global maximum of $J|_{\R^+u+\tX}$, and $u_2$ is the unique global maximum of $J|_{\R^-u+\tX}$. Moreover, $u_1$ and $u_2$ depend continuously on $u\in SX^+$. Setting $n(u^+):=u_1$ there holds $\cN=\{n(u^+):u^+ \in SX^+\}$.
\end{Prop}

\begin{Rem}\label{rem:nehari}\rm
  a) Clearly, $n(-u^+)=u_2$ in Proposition~\ref{prop:nehari}.

  b) It is possible that $I$ has a critical point $v\in\tX$. Then the energy $J(v)=-I(v)\le0$ is non-positive and $v\notin\cN$.
\end{Rem}

Since $J$ is not required to be $\cC^2$ the Nehari-Pankov manifold is just a topological manifold homeomorphic to $SX^+$.
The following result is due to \cite{BartschMederski1}.

\begin{Th}\label{ThLink1}
Suppose $J \in \cC^1(X,\R)$ satisfies (I1)-(I4), (I6), (I8), (I9), and suppose $J$ is coercive on $\cN$, i.e.\ $J(u)\to\infty$ as $\|u\|\to\infty$ and $u\in\cN$. Then the following holds:
\begin{itemize}
\item[a)] $c_{\cN} := \inf_\cN J\ge a>0$ and $J$ has a $(PS)_{c_{\cN}}$-sequence in $\cN$.
\item[b)] If $J$ satisfies the $(PS)_{c_{\cN}}^\cT$-condition in $\cN$ then $c_{\cN}$ is achieved by a critical point of $J$.
\item[c)] If $J$ satisfies the $(PS)_c^\cT$-condition in $\cN$ for every $c$ and if $J$ is even then it has an unbounded sequence of critical values.
\item[d)] If in addition (I5) holds then $c_{\cM}\leq c_{\cN}$, and if $c_{\cM}$ is achieved by a critical point then $c_{\cM}=c_{\cN}$.
\end{itemize}
\end{Th}

\begin{Rem}\label{rem:cNsubsetcM}
  a) If (I1)-(I6), (I8), (I9) hold then $\cN\subset\cM$ divides $\cM$ into two components. In fact, for $u\in SX^+$ there exists a unique $t_u>0$ such that $n(u)=t_uu+v$ with $v\in\wt X$. Then
  \[
  \cM\setminus\cN=\{m(tu):u\in SX^+,\, 0\le t<t_u\}\cup\{m(tu):u\in SX^+,\, t>t_u\}.
  \]
  The map $\be_u:[0,\infty)\to\R$ defined by $\be_u(t)=J(m(tu))$ achieves its maximum at $t_u>0$. If $\be_u'(t)=J'(m(tu))[u]=0$ then $J'(m(tu))|_{\R u\oplus\wt X}=0$, hence $m(tu)\in\cN$ and $t=t_u$. It follows that $\be_u(t)$ is strictly increasing on $[0,t_u]$ and strictly decreasing on $[t_u,\infty)$. Now it is easy to see that a mountain pass solution for $J\circ m$ corresponds to a minimizer of $J\circ n$.

  b) As mentioned after Theorem~\ref{thm:mpt} condition (I8) applies for $I$ being superquadratic. If $I$ is asymptotically quadratic one can still define $\cN$. Then $\cN$ cannot be parametrized over $SX^+$ but only over a subset of $SX^+$. If $\cN\ne\emptyset$ one can still obtain a critical point of $J$ via minimization over $\cN$. This has been done in \cite{Qin-Tang:2016}.
\end{Rem}

\section{The bounded domain case}\label{sec:res-bounded}

In this section we consider the curl-curl equation \eqref{eq:main} on a bounded Lipschitz domain $\Om$, and we present results from \cite{BartschMederski1,BartschMederski2,Qin-Tang:2016,Tang-Qin:2016}. Recall the hypotheses (L1), (F1)-(F3) from Section~\ref{sec:var} which imply that the functional
\[
  J(u) = \frac12\int_\Om\langle\mu(x)^{-1}\curlop u,\curlop u\rangle\, dx
            - \frac12\int_\Om \langle V(x)u, u\rangle\, dx - \int_\Om F(x,u)\,dx
\]
from \eqref{eq:action} is defined and of class $\cC^1$ on the space $X = H_{V,0}(\curl;\Om)\cap L^p(\Om,\R^3)$. Critical points of $J$ are weak solutions of \eqref{eq:main}. We also recall the Helmholtz decomposition $H_{V,0}(\curl;\Om)=\cV_0\oplus\cW_0$ from \eqref{eq:Helmholtz}, and the corresponding Helmholtz condition $X=\cV\oplus\cW$ where $\cV=\cV_0\cap X$ and $\cW=\cW_0\cap X$. We need one more assumption concerning this decomposition.
\begin{itemize}
\item[(L2)] $\cV_0$ is compactly embedded into $L^p(\Om,\R^3)$ for $p\in(2,6)$ from (F3).
\end{itemize}
This implies of course that $\cV=\cV_0$.

\begin{Rem}
  Recall that (L1) implies that $\cV_0$ embeds compactly into $L^2(\Om,\R^3)$ by \cite[Theorem~4.7]{Bauer-Pauly-Schomburg:2016}. It seems to be open whether (L1) implies (L2). Clearly (L2) follows if $\cV_0$ embeds into $H^1(\Om,\R^3)$. This has been proved for $V=\id_{3\times 3}$ and $\partial \Om$ of class $\cC^{1,1}$, or $\Om$ convex, in \cite[Theorems~2.12, 2.17]{Amrouche}. Costabel et al.\ \cite{CostabelDN1999} and Hiptmair \cite[Section 4]{Hiptmair} proved the embedding $\cV_0\subset H^1(\Om,\R^3)$ for Lipschitz domains admitting singularities and for isotropic and piecewise constant $V$. It also holds if $V$ is Lipschitz continuous and $\Om$ has $\cC^2$ boundary, as shown in \cite[Proposition~3.1]{BartschMederski2}. It is known that the space
  \[
  X_N(\Om) := \left\{E\in H_0(\curl;\Om): \div(E)\in L^2(\Om,\R^3)\right\}
  \]
  embeds continuously into $H^{\frac12}(\Om,\R^3)$, hence compactly into $L^p(\Om,\R^3)$ for $p<3$; see \cite[Theorem~2]{Costabel}.
\end{Rem}

We want to apply the critical point theory from Section~\ref{sec:CriticalTheory} to $J$. First we need to find a decomposition $X=X^+\oplus\tX$ so that $J(u)=\frac12\|u^+\|^2-I(u)$ is as in \eqref{EqJ}. This is of course determined by the quadratic part of $J$ given by the form
\[
  X=\cV\oplus\cW\to\R,\quad v+w \mapsto
    \int_\Om \left(\langle\mu(x)^{-1}\curlop v,\curlop v\rangle - \langle V(x)v,v\rangle
                   - \langle V(x)w,w\rangle\right)\,dx.
\]
Since $V>0$ the space $X^+\subset\cV$ is the positive eigenspace of the quadratic form
\[
  Q(v) = \int_\Om \left(\langle\mu(x)^{-1}\curlop v,\curlop v\rangle - \langle V(x)v,v\rangle\right)\,dx
\]
from \eqref{eq:DefofQ}, i.e.\ $X^+=\cV^+$ is the sum the eigenspaces of eigenvalues $\la_j>1$ of the curl-curl source eigenvalue problem \eqref{EgEigenvalue}. And $\tX=\tcV\oplus\cW$ where $\tcV$ is the (finite-dimensional) span of the eigenspaces of \eqref{EgEigenvalue} corresponding to the eigenvalues $\la_j\le1$. For $v\in\V$ we denote by $v^+\in\V^+$ and $\tv \in\tcV$ the corresponding summands such that $v=v^++\tv$. Then $J$ has the form
\[
  J(v+w) = \frac12 Q(v^+) + \frac12 Q(\tv) - \frac12\int_{\Om}\langle V(x)w,w\rangle\,dx - \int_\Om F(x,v+w)
\]
As a consequence on Lemma~\ref{lem:evp} we can define a new equivalent norm on $X^+$ by $\|v\|^2 := Q(v)$ so that $J$ has the form $J(v+w) = \frac12\|v^+\|^2 - I(v+w)$ as in \eqref{EqJ} with
\[
  \begin{aligned}
    I(v+w)
      &= -\frac12Q(\tv)^2 + \frac12\int_{\Om}\langle V(x)w,w\rangle\,dx + \int_\Om F(x,v+w)\\
      &= -\frac12Q(\tv)^2 + \frac12\|w\|_V^2 + \frac12\int_\Om\langle V(x)v,v\rangle\,dx
          + \int_\Om F(x,v+w).
  \end{aligned}
\]
Observe that $I$ is (strictly) convex if $F(x,u)$ is (strictly) convex in $u\in\R^3$ because $Q$ is negative semi-definite on $\tcV$. If $Q$ is even negative definite on $\tcV$, i.e.\ \eqref{EgEigenvalue} does not have an eigenvalue $1$, then $I$ is strictly convex provided $F$ is just convex in $u$.

Our first main result of this section deals with the case of a superlinear nonlinearity. We require two more assumptions in addition to (F1)-(F4).

\begin{itemize}
  \item[(F6)] There exists $\ga>2$ such that $\langle f(x,u),u\rangle \ge \ga F(x,u)\geq 0$ for all $ u\in\R^3$, and $\displaystyle \essinf_{x\in\Om,|u|=r} F(x,u)>0$ for some $r>0$.
  \item[(F7)] $F(x,u)$ is convex in $u$ for a.e.\ $x\in\Om$, and strictly convex in $u$ for a.e.\ $x$ if $1$ is an eigenvalue of \eqref{EgEigenvalue}.
\end{itemize}

Observe that (F6) implies
\begin{itemize}
  \item[(F8)] $F(x,u)\geq 0$ for a.e. $x\in\Om$, $u\in\R^3$ and there exists a constant $d>0$ such that $\displaystyle \liminf_{|u|\to\infty}\frac{F(x,u)}{|u|^\ga} \ge d$ for a.e.\ $x\in\Om$.
\end{itemize}

\begin{Th}\label{thm:main1}
Suppose (L1)-(L2) and (F1)-(F3), (F6)-(F7) hold.

a) Equation \eqref{eq:main} has a nontrivial solution $u\in X$.

b) If $F$ is even in $u$ then \eqref{eq:main} has a sequence of solutions $u_n$ with $J(u_n)\to\infty$.
\end{Th}

The theorem follows from Theorem~\ref{thm:mpt}. Details of the proof of the assumptions (I1)-(I8) of Theorem~\ref{thm:mpt} can be found in \cite[Section~5]{BartschMederski2}. In \cite{BartschMederski2} it was even allowed that $F(x,u)=0$ for $|u|>0$ small. This models materials where the polarization is linear if the intensity of the electric field $\cE$ is small. There we also address the existence of a ground state being defined as a solution of \eqref{eq:main} with positive energy that has the least energy among all solutions with positive energy. This requires an additional condition on $F$.

\begin{Rem}\rm
If $u=v+w\in X$ is a nontrivial solution of \eqref{eq:main} with $v\in\cV$ and $w\in\cW$ then necessarily $v^+\ne0$. This is a simple consequence of (L1) and (F6)-(F7). In fact, testing \eqref{eq:main} with $v+w$ and using the positivity of $V$ as well as the convexity of $F$ yields:
\[
  Q(v^+) = \int_\Om \langle V(x)(\tv+w),\tv+w\rangle\,dx + \int_\Om \langle f(x,v+w),v+w\rangle\,dx > 0.
\]
\end{Rem}

As mentioned in Section~\ref{sec:CriticalTheory} the approach via the Nehari-Pankov manifold requires another type of condition that we discuss next.
\begin{itemize}
  \item[(F9)] $\frac{t^2-1}{2}f(x,u)[u] + tf(x,u)[v] + F(x,u) - F(x,tu+v) \le 0$ for all $t\ge0$, $u,v\in\R^3$, a.e.\ $x\in\Om$, and the strict inequality holds if $u\ne tu+v$.
\end{itemize}

\begin{Rem}\label{rem:F9}\rm
  a) Versions of (F9) appear in \cite[Proof of Lemma~5.2]{BartschMederski1}, \cite[Condition~(B3)]{BartschMederski2}, \cite[Condition~(F4)]{Qin-Tang:2016}, \cite[Lemma~38]{SzulkinWethHandbook}, \cite[Condition~(F7')]{Tang-Qin:2016}.

  b) In the scalar case $N=1$ it has been proved in \cite[Lemma~38]{SzulkinWethHandbook} that (F9) follows from $f(x,u)=o(u)$ as $u\to0$ and:
  \[
    \R\ni u \mapsto \frac{f(x,u)}{|u|}\in\R \quad\text{is strictly increasing on $(-\infty,0)$ and on $(0,\infty)$.}
  \]
  This is the typical condition for setting up the classical Nehari manifold. For $N\ge2$ no version of this condition is known to imply (F9).

  c) Observe that $F:\Om\times\R^3\to\R$ is strictly convex in $u\in\R^3$ if (F9) holds. In order to see this consider the map $g(s):=F(x,(1-s)u_0+su_1)$ for given $u_0\ne u_1\in\R^N$, $x\in\Om$. Then (F9) yields for $0\le s<r\le1$:
  \[
    \begin{aligned}
      g'(s)(r-s) &= f(x,(1-s)u_0+su_1)[(r-s)(u_1-u_0)]\\
                 &< F(x,(1-r)u_0+ru_1) - F(x,(1-s)u_0+su_1) = g(r)-g(s).
    \end{aligned}
  \]
  Here we applied (F9) with $t=1$, $u=(1-s)u_0+su_1$, $v=(r-s)(u_1-u_0)$. This implies the strict convexity of $g$, hence of $F$. As a consequence, if (F9) holds then the functional $J(u)=\frac12 \|u^+\|^2-I(u)$ with $I(u)=\int_\Om F(x,u(x))\,dx$ satisfies (I5).

  d) In \cite{BartschMederski1} the following condition has been used instead of (F9):
  \begin{itemize}
    \item[(*)] If $ f(x,u)[v] = f(x,v)[u] > 0\ $ then
     $\ \displaystyle F(x,u) - F(x,v) \le \frac{(f(x,u)[u])^2-(f(x,u)[v])^2}{2f(x,u)[u]}$.\\
     If in addition $F(x,u)\ne F(x,v)$ then the strict inequality holds.
  \end{itemize}
  Actually in \cite{BartschMederski1} the condition was a bit stronger in that the weak inequality in (*) was also required if $f(x,u)[v] = f(x,v)[u] < 0$. However for the proof of \cite[Lemma~5.2]{BartschMederski1} only (*) is needed. It has been proved in \cite[Lemma~5.2]{BartschMederski1} that (F1), (F2), (*) and
  \begin{itemize}
    \item[(**)]  $F$ is strictly convex in $u\in \R^3$, and for any $x\in\R^3$, any $u\in\R^3$, $u\neq 0$:
	\[ F'(x,u)[u]> 2 F(x,u). \]
  \end{itemize}
  imply (F9). On the other hand, (F1), (F2), (F9) and (**) imply (*). Indeed, suppose that $f(x,u)\ne 0$, hence $f(x,u)[u]>0$ by the convexity of $F$. Then setting $w=-tu+v$ and $t=\frac{f(x,u)[v]}{f(x,u)[u]}>0$, we get
  \[
  \begin{aligned}
    &\frac{(f(x,u)[v])^2-(f(x,u)[u])^2}{2f(x,u)[u]} + F(x,u) - F(x,v) \\
    &\hspace{1cm}
      = -\frac{t^2+1}{2}f(x,u)[u] + tf(u)[v] + F(x,u) - F(x,v) \\
    &\hspace{1cm}
      = \frac{t^2-1}{2}f(x,u)[u] + tf(x,u)[w] + F(x,u) - F(x,tu+w) \le 0.
  \end{aligned}
  \]
  In particular, if $\langle f(x,u),v\rangle=\langle f(x,v),u\rangle>0$, then $f(x,u)\neq 0$ and by (**), $f(x,u)u>0$. Moreover, if in addition $F(x,u)\neq F(x,v)$ then $u\neq v= tu+(-tu+v)$ and by (F9) the strict inequality holds.
\end{Rem}

\begin{Th}\label{thm:main2}
  Suppose (L1)-(L2) and (F1)-(F3), (F8)-(F9) hold.

  a) Equation \eqref{eq:main} has a ground state solution $u\in X$, i.e.\ the minimum of the functional $J$ on the Nehari-Pankov manifold $\cN$ is achieved.

  b) If $F$ is even in $u$ then \eqref{eq:main} has a sequence of solutions $u_n$ with $J(u_n)\to\infty$.
\end{Th}

Theorem~\ref{thm:main2} is a consequence of Theorem~\ref{ThLink1}; see \cite[Section~5]{BartschMederski2} for the proof of properties (I1)-(I4), (I6)-(I8), and the coerciveness of $J$ on $\cN$.

\begin{Rem}
  The first result for \eqref{eq:main} in a bounded domain with superlinear nonlinearity and metallic boundary condition \eqref{eq:bc} is due to \cite{BartschMederski1}. There $V(x)\equiv\la>0$ was a constant scalar, and $f(x,u)=f(u)$ was independent of $x\in\Om$. The proof was based on the Nehari-Pankov manifold approach. A closely related result using the same approach is due to \cite{Tang-Qin:2016}. In these papers one can find several classes of functions that satisfy the hypotheses, in particular (F9). The model nonlinearity is $F(x,u)=\Ga(x)|u|^p$ with $2<p<6$. Also sums of such functions are allowed.
\end{Rem}

Now we discuss asymptotically linear nonlinearities. So far these have only been considered in \cite{Qin-Tang:2016}. Here we present a variation of their main result \cite[Theorem~1.1]{Qin-Tang:2016} within our setting. The main difference is that $V(x)\equiv\la>0$ is a constant scalar in \cite{Qin-Tang:2016}.
\begin{itemize}
  \item[(F10)] There exists $V_\infty\in L^{\infty}(\Om,\R^{3\times 3})$ such that $f(x,u)=V_\infty(x)[u]+f_\infty(x,u)$ for a.e.\ $x\in\Om$, every $u\in\R^3$. Moreover, $\langle f(x,u),f_\infty(x,u)\rangle < 0$ for $u\ne0$, and $|f_\infty(x,u)|=o(|u|^\si)$ as $|u|\to\infty$ for some $\si\in(0,1)$ uniformly in $x\in\Om$.
  \item[(F11)] $\frac12\langle f(x,u),u\rangle - F(x,u) \to \infty$ as $|u|\to\infty$ uniformly in $x\in\Om$.
\end{itemize}

Similar to Theorem~\ref{thm:sym1-asy-lin} we consider the quadratic forms
\[
  Q_0(v) := \int_\Om \left(\langle\mu(x)^{-1}\curlop v,\curlop v\rangle - \langle V(x)v,v\rangle \right)\,dx,
\]
and
\[
  Q_\infty(v)
   := \int_\Om \left(\langle\mu(x)^{-1}\curlop v,\curlop v\rangle - \langle (V(x)+V_\infty(x))v,v\rangle \right)\,dx,
\]
on the space $\cV$.

\begin{Th}\label{thm:asy-lin}
  Suppose (L1), (F1), (F2), (F9)-(F11) hold. Moreover suppose that the quadratic forms $Q_0,Q_\infty:V\to\R$ are nondegenerate with indices $i_0,i_\infty\in\N_0$, respectively. Then $i_0\le i_\infty$, and if $i_0<i_\infty$ there exists a nontrivial solution $u\in X$ of \eqref{eq:main}. If in addition $F$ is even in $u$ then \eqref{eq:main} has at least $i_\infty-i_0$ different pairs of solutions $\pm u_n$.
\end{Th}

Since the nonlinearity $F$ is asymptotically quadratic the functional is defined on $X=H_{V,0}(\curl;\Om)$ with Helmholtz decomposition $\cV_0\oplus\cW_0$. Since $\cV_0$ embeds compactly into $L^2_V(\Om,\R^3)$ we do not need to assume (L2).  Recall that Remark~\ref{rem:F9}~b) implies that $F$ is convex in $u$ as a consequence of (F9). Then the linearization $V_\infty$ of $f$ in (F2) must be positive semi-definite, so $Q_\infty\le Q_0$, hence $i_0\le i_\infty$.
The main observation from \cite{Qin-Tang:2016} is that the Nehari-Pankov manifold $\cN$ from \eqref{eq:NehariDef} is not homeomorphic to the unit sphere $SX^+$ but only to an open subset $\cO\subset SX^+$ which can be explicitly determined. The homeomorphism is given by the map $n:\cO\to\cN$ from \eqref{eq:def-n}. Let $\cV_\infty^-\subset\cV$ be the negative eigenspace associated to $Q_\infty$, so that $\dim(X^+\cap\cV_\infty^-)\ge i_\infty-i_0$. Then one can show that $\cO$ contains the set $S(X^+\cap\cV_\infty^-):=\{u\in X^+\cap\cV_\infty^-: \|u\|=1\}$. As before it is sufficient to find critical points of $J\circ n$. This map is bounded below and satisfies the Palais-Smale condition. Now Theorem~\ref{thm:asy-lin} follows by minimizing $J$ on $\cN$ in order to obtain the ground state, and from standard Lusternik-Schnirelmann theory for the multiple solutions.

\begin{Rem}\label{rem:sat}
  As in the symmetric case (see Remark~\ref{rem:sym1-sat}) Theorem~\ref{thm:asy-lin} applies in particular to the nonlinearity with saturation from \eqref{ex:sat}. This is a special case of a more general nonlinearity of the form $F(x,u)=\frac12 H(x,|u|^2)$ where $h=\pa_tH:\Om\times\R^+\to\R$ is a Carath\'eodory function satisfying appropriate conditions so that (F1), (F2), (F9)-(F11) hold; see the discussion in \cite{Qin-Tang:2016} after Theorem~1.1 where this class of examples has been presented.
\end{Rem}

We conclude this chapter with a short discussion of the symmetric situation, i.e.\ when assumption (S) from Section~\ref{sec:sym} holds in addition to the other assumptions of Theorems~\ref{thm:main1}, \ref{thm:main2}, \ref{thm:asy-lin}. We are especially interested in the existence of solutions of the form \eqref{eq:sym2}, i.e.\
\[
  u(x)=\be(r,x_3)\begin{pmatrix}x_1\\x_2\\0\end{pmatrix} + \ga(r,x_3)\begin{pmatrix}0\\0\\1\end{pmatrix}.
\]
Recall the action of the group $G=O(2)\times\{1\}\subset O(3)$ on $X$ from \eqref{eq:G-action} and the linear isometry $S_1:X^G\to X^G$ from \eqref{eq:def-S1}. Recall also that the fixed point space $(X^G)^{S_2}$ consists of fields of the form \eqref{eq:sym2}. Assumption (S) implies that $J$ is invariant under $G$ and under $S_2=-S_1$, hence it suffices to find critical points of $J|(X^G)^{S_2}$.

\begin{Th}\label{thm:sym2}
  Suppose in Theorems~\ref{thm:main1}, \ref{thm:main2} that the symmetry assumption (S) holds in addition to the other assumptions. Then these theorems remain true and yield solutions in $(X^G)^{S_1}$ and in $(X^G)^{S_2}$. The least energy solutions in Theorem~\ref{thm:main2} can be obtained by minimization on the Nehari-Pankov manifold in $(X^G)^{S_k}$, $k=1,2$. In the case of \ref{thm:asy-lin} the quadratic forms $Q_0,Q_\infty$ have to be considered on the space $X^+\cap (X^G)^{S_k}$, $k=1,2$.
\end{Th}

\section{The case $\Om=\R^3$}\label{sec:R^3}

The first results about solutions of
\begin{equation}\label{eq:mainR3}
  \curlop\curlop u - V(x)u = f(x,u) = \nabla_uF(x,u) \qquad\text{in $\R^3$,}
\end{equation}
are due to \cite{BenForAzzAprile, DAprileSiciliano} in the case $V=0$ and $F(x,u)=\frac12W(|u|^2)$ where $W:\R^+\to\R$ grows supercritically for $|u|\le1$ and subcritically for $|u|>1$; cf.\ condition (F14) below. Equation \eqref{eq:mainR3} then has the form \eqref{eq:BF}. Clearly such $F$ satisfies the symmetry condition (S) and is even. In view of Theorem \ref{thm:radial}, the assumption of full radial symmetry does not lead to interesting solutions of \eqref{eq:mainR3}. Solutions of the form \eqref{eq:sym1} are obtained in \cite{BenForAzzAprile}, solutions of the form \eqref{eq:sym2} in \cite{DAprileSiciliano}, of course under appropriate hypotheses on $W$. The basic approach in these papers consists in minimizing $J$ on the constraint $\{u: W(|u|^2)=1\}$ in a suitable subspace of $\D^{1,2}(\R^3,\R^3)$. Such a minimizer $v$ leads to a Lagrange multiplier in the equation that can be scaled away by considering $u(x)=v(\al x)$ for a certain choice of $\al>0$. This does not work in the nonautonomous case, hence we just refer the reader to \cite{BenForAzzAprile,DAprileSiciliano} for details of this approach.

\subsection{Results in the cylindrically symmetric setting}

Throughout this section we impose the symmetry condition (S), i.e.\ $V=(r,x_r)$ and $F=F(r,x_3)$. We also recall the following conditions on the nonlinear term $F$:
\begin{itemize}
  \item[(F1)] $F:\R^3\times\R^3\to\R$ is differentiable with respect to $u\in\R^3$, such that $f=\nabla_uF:\R^3\times\R^3\to\R^3$ is a Carath\'eodory function (i.e.\ measurable in $x\in\Om$, continuous in $u\in\R^3$ for a.e.\ $x\in\R^3$). Moreover, $F(x,0)=0$ for a.e.\ $x\in\R^3$.
  \item[(F2)] $|f(x,u)|=o(|u|)$ as $u\to0$ uniformly in $x\in\R^3$.
  \item[(F3)] There exist $2<p<6$ and a constant $c>0$ such that
	\[
	  |f(x,u)|\le c(1+|u|^{p-1})\qquad\text{for all } x\in\R^3, u\in\R^3.
	\]
  \item[(F9)]$\frac{t^2-1}{2}F'(x,u)[u] + tF'(x,u)[v] + F(x,u) - F(x,tu+v) \leq 0$ for all $t\ge0$, $u,v\in\R^3$, a.e.\ $x\in\R^3$, and the strict inequality holds if $u\neq tu+v$.
  \item[(F12)] $F(x,u)/|u|^2\to\infty$ as $|u|\to\infty$ uniformly in $x\in\R^3$.
\end{itemize}

We shall look for critical points of
\begin{equation}\label{eq:actionR3}
  J(u) = \frac12\int_{\R^3}|\curlop u|^2\, dx - \frac12\int_{\R^3} V(x)|u|^2\, dx - \int_{\R^3} F(x,u)\,dx
\end{equation}
defined on the space
\begin{equation}\label{eq:def-Xp}
  X:=\D(\curl,p)\cap L^2_{|V|}(\R^3,\R^3),
\end{equation}
where $\cD(\curl,p)$ is the completion of $\cC_0^{\infty}(\R^3,\R^3)$ with respect to the norm
\[
  \|u\|_{\curl,p}:=(|\curlop u|^2_2+|u|_p^2)^{1/2}
\]
and $L^2_{|V|}(\R^3,\R^3)$ is the space of square integrable vector fields with respect to the measure $|V|\,dx$.

Since (S) holds we consider as in Section~\ref{sec:sym} the action of $G=O(2)\times\{1\}\subset O(3)$ on $X$ from \eqref{eq:G-action}, and the isometry $S_1:X^G \to X^G$ from \eqref{eq:def-S1}. If $F$ is in addition even in $u$ then as before it is sufficient to find critical points of $J$ constrained to the fixed point set
\[
  \begin{aligned}
    (X^G)^{S_1} &:= \{u\in X^G: S_1(u)=u\} = \{u\in X^G: u=u_\tau\}\\
                &\subset \{u\in H(\curl;\R^3): \div(u) =0\} = \{u\in H^1(\R^3,\R^3): \div(u) =0\}.
  \end{aligned}
\]
Similarly we can define subspaces
\[
  (H^2(\R^3,\R^3)^G)^{S_1}\subset \{u\in H^2(\R^3,\R^3): \div(u) =0\}
\]
and
\[
  (L^2(\R^3,\R^3)^G)^{S_1}\subset L^2(\R^3,\R^3).
\]

If $V\in L^\infty(\R^3)$ then in view of \cite[Lemma 4.4]{BDPR:2016}, the operator $\cL := (\curlop\curlop)+V$ defined on
\[
  D(\cL) = (H^2(\R^3,\R^3)^G)^{S_1}\subset (L^2(\R^3,\R^3)^G)^{S_1}\to (L^2(\R^3,\R^3)^G)^{S_1}
\]
is selfadjoint.

We also require the following periodicity condition.
\begin{itemize}
  \item[(P)] $V$ and $F$ are $1$-periodic in $x_3$, i.e.\ $V(r,x_3)=V(r,x_3+1)$, $F(r,x_3,|u|,u_3)=F(r,x_3+1,|u|,u_3)$ for a.a. $r>0, x_3\in \R$, $u\in\R^3$.
\end{itemize}

\begin{Th}\label{Th1SymR^3}
  Assume that (S), (P) hold, that $V\in L^\infty(\R^3)$, and that $F$ is even in $u$ and satisfies (F1)-(F3), (F9), (F12). If $0\notin \si(\cL)$ then the equation \eqref{eq:mainR3} has a ground state solution in $(X^G)^{S_1} $, which is a minimizer of $J$ on the associated Nehari-Pankov manifold $\cN\subset (X^G)^{S_1} $.
\end{Th}

The proof is based on Theorem~\ref{ThLink1}~a). One works on the space $X$ with $p$ from (F3). It is sufficient to find critical points of $J$ constrained to $Y:=(X^G)^{S_1}$. Since $0\not \in \si(\cL)$ there is a direct sum decomposition $(X^G)^{S_1} = X^+\oplus\tX$ such that the quadratic form $Q(u) := \int_{\R^3}|\curlop u|^2 -  V(x)|u|^2\,dx$ is positive on $X^+$, negative on $\tX$, and $\|u\|=(Q(u^+)-Q(\tu))^{1/2}$ defines an equivalent norm for $u=u^++\tu\in (X^G)^{S_1}=X^+\oplus\tX$. Now observe that $J:(X^G)^{S_1}\to \R$ has the form $J(u)=\frac12\|u^+\|^2-I(u),$ as in \eqref{EqJ} with $I(u)=Q(\tu)+\int_{\R^3}F(x,u)\,dx$. One checks that $J$ is coercive on $\cN$ and satisfies  (I1)-(I4), (I6), (I8), (I9), where
\[
  \cN= \{u\in (X^G)^{S_1}\setminus\tX: J'(u)|_{\R u\oplus \tX}=0\}.
\]
Theorem~\ref{ThLink1}~a) yields a Palais-Smale sequence  $(u_n)\subset \cN$ that minimizes $J$ on $\cN$. Then, by means of a concentration-compactness argument, one shows that, after passing to a subsequence and up to $\Z$-translation in $x_3$, $u_n$ is weakly convergent to a nontrivial ground state of $J$; cf.\ \cite[Proof of Theorem~1.3]{BDPR:2016}. We leave details for the reader.

\begin{Rem}
  a) In the special case $F(x,u)=\Ga(x)|u|^p$ with $\Ga=\Ga(r,x_3)\in L^{\infty}(\R^3)$ cylindrically symmetric, periodic in $x_3$ and bounded away from $0$, the above result has been obtained in \cite[Theorem~1.3]{BDPR:2016}. For such $F$ the Nehari-Pankov manifold is of class $\cC^1$ which makes the proof somewhat easier. In \cite[Section~4]{BDPR:2016} one can also find an example for a potential $V=V(r,x_3)$ satisfying (P) with $0\not \in \si(\cL)$ and $\si(\cL)\cap(0,\infty)\ne\emptyset$.

  b) Since $(X^G)^{S_1}$ is locally compactly embedded into $L^p(\R^3,\R^3)$ it is easy to check that $J'$ is weak-to-weak$^*$ continuous in $(X^G)^{S_1}$, and one can apply as well linking results from \cite{BenciRabinowitz,Bartsch:1993,MederskiSystem} varying the hypotheses on $F$. For instance, condition (F9) can be weakened by not requiring the strict inequality, and Theorem \ref{Th1SymR^3} remains valid. Indeed, by means of \cite[Theorem~2.1]{MederskiSystem} one finds a Cerami sequence $(u_n)$ at level $0<c\leq \inf_{\cN}J$.  Again, by a concentration-compactness argument, one sees that, after passing to a subsequence and up to $\Z$-translation in $x_3$, $u_n$ is weakly convergent to a nontrivial critical point $u_0$ of $J$. It follows that $u_0\in\cN$ is a ground state. Detailed arguments can be provided as in \cite{MederskiSystem}. The ground state level may also be characterized in terms via an infinite-dimensional min-max scheme in $(X^G)^{S_1}$ as in \cite{MederskiSystem}[(1.5)].

  c) The assumption $0\notin \si(\cL)$ excludes the case $V=0$ treated in \cite{BenForAzzAprile}.
\end{Rem}

Under additional assumptions on $V$ and $F$ more can be said about the symmetry of the ground state. The following result is a special case of \cite[Theorem~1]{HirschReichel}.

\begin{Th}\label{thm:symR3}
  Assume that $V=V(r,x_3)$, $F=\Ga(r,x_3)|u|^p$ with $2<p<6$, so that $V,\Ga\in L^\infty(\R^3)$ and $\inf V,\inf \Ga>0$. Suppose moreover that $V$ and $\Ga$ are Steiner symmetric in $x_3$. Then the ground state solution in $(X^G)^{S_1}$ of \eqref{eq:mainR3} from Theorem~\ref{Th1SymR^3} is symmetric about the plane $\{x_3=0\}$.
\end{Th}

At the end of this subsection we mention a result from \cite{BDPR:2016}[Theorem 1.2] for a defocusing nonlinearity.

\begin{Th} Let $F(x,u)=\Ga(x) |u|^p$ with $p>2$ and assume that $V=V(r,x_3)$ and $\Ga=\Ga(r,x_3)$ have cylindrical symmetry. Suppose moreover:
	\begin{itemize}
		\item[(i)] $\Ga(x) \leq -C(1+|x|)^\al$ in $\R^3$ with $\al > \frac{3}{2}p$ and $C>0$,
		\item[(ii)] $V\in L^\infty(\R^3)$ and  $\esssup_{\R^3} V<0$.
	\end{itemize}
	Then \eqref{eq:mainR3} has a solution that is a global minimizer of $J$ in $(X^G)^{S_1}$.
\end{Th}

The proof is a minimization argument in $(X^G)^{S^1}$.

\subsection{The general case}

In this subsection we do not require any symmetry assumptions and we allow that $V=0$. Therefore even if (S) holds the results will be different from those of the last subsection. We require the following conditions on $V:\R^3\to\R$ and $F:\R^3\times\R^3\to\R$:
\begin{itemize}
  \item[(V)] $V\in L^{\frac{p}{p-2}}(\R^3)\cap L^{\frac{q}{q-2}}(\R^3)$, $V(x)\geq 0$ for a.e.\ $x\in\R^3$ and $|V|_{\frac{3}{2}}<S$, where $S$ is the classical best Sobolev constant of the embedding of $\cD^{1,2}$ into $L^6(\R^3)$.
  \item[(F13)] $F$ is $\Z^3$-periodic in $x$, i.e. $F(x,u)=F(x+y,u)$ for $x,u\in \R^3$ and $y\in\Z^3$.
  \item[(F14)] There are $2<p<6<q$ and constants $c_1,c_2>0$ such that
	  \[ F(x,u)\geq c_1 \min(|u|^{p},|u|^{q}) \]
	  and
	  \[ |f(x,u)|\leq c_2 \min(|u|^{p-1},|u|^{q-1})\]
	  for all $x,u\in\R^3$.
  \item[(F15)] If $V=0$ a.e.\ on $\R^3$ then $F$ is uniformly strictly convex with respect to $u\in\R^3$, i.e.\ for any compact $A\subset(\R^3\times\R^3)\setminus\{(u,u):\;u\in\R^3\}$:
	  \[
        \inf_{\genfrac{}{}{0pt}{}{x\in\R^3}{(u_1,u_2)\in A}}
	      \left(\frac12\big(F(x,u_1)+F(x,u_2)\big)-F\left(x,\frac{u_1+u_2}{2}\right)\right) > 0.
	  \]
\end{itemize}
Clearly (F14) implies (F2) and (F12). Model nonlinearities are
\begin{equation}\label{Ex1}
F(x,u)=\left\{
\begin{array}{ll}
\Ga(x)\big(\frac1p|Mu|^p+\frac1q-\frac1p\big)
&
\hbox{if } |Mu|>1,\\
\Ga(x) \frac1q|Mu|^q
&
\hbox{if } |Mu|\leq 1,\\
\end{array}
\right.
\end{equation}
and
\begin{equation}\label{Ex2}
  F(x,u)=\Ga(x)\frac1p\big((1+|Mu|^q)^{\frac{p}{q}}-1\big)
\end{equation}
where $\Ga\in L^{\infty}(\R^3)$ is $\Z^3$-periodic and $\inf\Ga>0$, $M\in GL(3)$ is an invertible $3\times 3$ matrix. Then the assumptions (F1), (F9), (F13)-(F15) are satisfied. Observe that these functions are not radial when $M$ is not an orthogonal matrix.

We need the space $L^{p,q}:=L^p(\R^3,\R^3)+L^q(\R^3,\R^3)$ with the norm
\[
  |u|_{p,q} = \sup\left\{\frac{\int_{\R^3}\langle u, v\rangle\, dx}{|v|_{\frac{p}{p-1}}+|v|_{\frac{q}{q-1}}}:
                  v\in L^{\frac{p}{p-1}}(\R^3,\R^3)\cap L^{\frac{q}{q-1}}(\R^3,\R^3),\ v\neq 0\right\}.
\]
It coincides with the usual Lebesgue space $L^p(\R^3,\R^3)$ if $p=q$; see \cite{BadialePisaniRolando} for more properties of $L^{p,q}$. The functional $J$ is defined and of class $\cC^1$ on the completion $X=\D(\curl,p,q)$ of $\cC_0^{\infty}(\R^3,\R^3)$ with respect to the norm
where
\[
  \|u\|_{\curl,p,q}:=(|\curlop u|^2_2+|u|_{p,q}^2)^{1/2}.
\]
Any $u\in \cD(\curl,p,q)$ has the Helmholtz decomposition $u=v+\nabla w$ with $\div(v)=0$ and $\nabla w\in L^{p,q}$; see \cite[Lemma 3.2]{MederskiENZ}. Hence
\[
\begin{aligned}
J(v+\nabla w)
 &= \frac12\int_{\R^3}|\curlop v|^2\,dx - \frac12\int_{\R^3} V(x)|v+\nabla w|^2\,dx
     - \int_{\R^3} F(x,v+\nabla w)\,dx\\
 &= \frac12\int_{\R^3}|\nabla v|^2\,dx - \frac12\int_{\R^3} V(x)|v+\nabla w|^2\,dx - \int_{\R^3} F(x,v+\nabla w)\,dx
\end{aligned}
\]
and the Nehari-Pankov manifold is given by
\begin{equation}\label{DefOfNehari1}
	\cN := \{u\in \D(\curl,p,q)\setminus \W: J'(u)|_{\R u\oplus \W}=0\},
\end{equation}
where $\W$ is the closure of $\{\nabla\vp: \vp\in \cC_0^{\infty}(\R^3)\}$ in $\D(\curl,p,q)$. We would like to mention that $J$ with the above nonlinearities has the linking geometry; see \cite[Proposition 2.1]{MederskiSurv}.

The following result is due to \cite[Theorem~2.1]{MederskiENZ}.

\begin{Th}\label{ThMainNonsym}
  Assume that (V) and (F1), (F9), (F13)-(F15) hold. Then there is a solution to \eqref{eq:mainR3}. If $V<0$ a.e.\ on $\R^3$ or $V=0$ then \eqref{eq:mainR3} has a ground state solution, i.e.\ there is a critical point $u\in\cN$ which is a minimizer of $J$ restricted to $\cN$.
\end{Th}

\begin{Rem}
  In the cylindrically symmetric setting the result holds true with $X=\D(\curl,p,q)$ replaced by $(X^G)^{S_2}$, where $S_2=-S_1$, that is one obtains a ground state solution as minimizer of $J$ on $\cN\cap(X^G)^{S_2}$. If in addition $F$ is even in $u$ the result holds true with $X$ replaced by $(X^G)^{S_1}$. For $V=0$ and $F$ independent of $x$ one therefore recovers the main results from \cite{BenForAzzAprile} and \cite{DAprileSiciliano}.
\end{Rem}

A crucial role in the proof of Theorem \ref{ThMainNonsym} is a careful analysis of bounded sequences in $\cN$ which we recall now. We need the functional
\begin{equation}\label{DefOfI}
  I(u):= \frac12\int_{\R^3} V(x)|u|^2\,dx + \int_{\R^3} F(x,u)\,dx.
\end{equation}

\begin{Th}\label{ThMainSplitting}
  Assume that (V) and (F1), (F9), (F13)-(F15) hold. If $(u_n)_{n=0}^{\infty}\subset \cN$ is bounded then, up to a subsequence, there is $N\in \N\cup \{\infty\}$, $\bar{u}_0\in \D(\curl,p,q)$ and there are sequences $(\bar{u}_i)_{i=1}^N\subset \cN_0$ and $(x_n^i)_{n\geq i}\subset \Z^3$ with $x_n^0=0$ such that the following conditions hold:
  \begin{equation}\label{EqThMainSplitting1}
	u_n(\cdot+x_n^i)\rightharpoonup \bar{u}_i\hbox{ in }\D(\curl,p,q)\hbox{ and }
    u_n(\cdot+x_n^i)\to \bar{u}_i\hbox{ a.e. in }\R^3\hbox{ as }n\to\infty,
  \end{equation}
  for any $0\leq i < N+1$, and
  \begin{equation}\label{EqThMainSplitting2}
	u_n-\sum_{i=0}^{\min\{n,N\}}\bar{u}_i(\cdot-x_n^i)\to 0\hbox{ in }L^{p,q}=L^p(\R^3,\R^3)+L^q(\R^3,\R^3)
    \hbox{ as }n\to\infty.
  \end{equation}
  Moreover
  \begin{equation}\label{eqInfinitesplitting}
  \lim_{n\to\infty}I(u_n)=I(\bar{u}_0)+\sum_{i=1}^N I_0(\bar{u}_i)<\infty,
  \end{equation}
  where $\cN_0$ and $I_0$ are given by \eqref{DefOfNehari1} and \eqref{DefOfI} under assumption $V=0$.
\end{Th}

The proof can be found in \cite[Theorem 2.2]{MederskiENZ}.  As a consequence of Theorem \ref{ThMainSplitting} we obtain the weak-to-weak$^*$ continuity of $J'$ on $\cN$; see \cite[Corollary 5.3]{MederskiENZ}. Moreover, in the spirit of the global compactness result of Struwe \cite{StruweSplitting} or Coti Zelati and Rabinowitz \cite{CotiZelatiRab}, we obtain a finite splitting of energy levels with respect to a Palais-Smale sequence in $\cN$.

\begin{Th} Assume that (V) and (F1), (F9), (F13)-(F15) hold.
	If $(u_n)_{n=0}^{\infty}\subset \cN$ is a $(PS)_c$-sequence at level $c>0$, i.e. $J(E_n)\to c$ and $J'(E_n)\to 0$, then, up to a subsequence,
	there is $\bar{u}_0\in \D(\curl,p,q)$ and a finite sequence $(\bar{u}_i)_{i=1}^N\subset \cN_0$ of critical points of $\J_0$ such that \eqref{EqThMainSplitting1}, \eqref{EqThMainSplitting2}  hold and
	\begin{equation*}
	c=J(\bar{u}_0)+
	\sum_{i=1}^N J_0(\bar{u}_i),
	\end{equation*}
	where $J_0$ is the energy functional given by \eqref{eq:actionR3} under assumption $V=0$.
\end{Th}

Now, observe that if $0<c<\inf_{\cN_0} J_0$ then $N=0$, $J(\bar{u}_0)=c$ and $\bar{u}_0$ is a nontrivial critical point of $J$. In this way the comparison of energy levels will imply the existence of nontrivial solutions. See \cite{MederskiENZ} for detailed proof of Theorem \ref{ThMainNonsym}.

\section{Open problems}\label{sec:problems}

{\bf Problem 1.} What can one say about the symmetry of the ground state solution if (S) holds? We conjecture that if $\Om$ is a ball, then the ground state solution $u\in H_0(\curl;\Om)\cap L^p(\Om,\R^3)$ of
\[
  \curlop\curlop u + u = |u|^{p-2}u
\]
with $2<p<6$ is symmetric in the sense of \eqref{eq:sym1}. For other domains, e.g.\ $\Om=\{x\in\R^3:1<x_1^2+x_2^2<2,\,0<x_3<1\}$, a symmetry breaking might occur depending on $p$. It would also be very interesting to find criteria on $\Om$, $V(x)$, $F(x,u)$ satisfying (S) so that the general non-autonomous problem \eqref{eq:main} has a ground state that is invariant under $G=O(2)\times\{1\}$, or has the form \eqref{eq:sym1} if $F$ is even in $u$, or has the form \eqref{eq:sym2}. Of course the problem is also very interesting for unbounded domains and for critical nonlinearities. \\

{\bf Problem 2.} Suppose that $V=\la\id_{3\times 3}$ with $\la\in\R$, $\mu=\id_{3\times 3}$  and $f(x,u)=|u|^{p-2}u$ with $2<p<6$. In view of Theorem \ref{thm:main2} there is a ground state solution for $\la\ge 0$ on a Lipschitz domain with $\cC^{1,1}$ boundary. If in addition $\Om$ is cylindrically symmetric, i.e.\ (S) holds, then there are solutions of the form \eqref{eq:sym1} for all $\la\in\R$; see Theorem \ref{thm:sym1-superlin}. It is an open problem to show the existence of solutions for $\la<0$ and nonsymmetric domain $\Om$. Setting
\[
\la_0= \mu(\Om)^{-\frac{p-2}{p}}p^{-\frac{2}{p}}\inf_{v\in \V:\; |v|_{p}=1}\int_{\Om} |\nabla\times u|^2\; dx>0,
\]
we observe that $J$ has the linking geometry for $\la>-\la_0$. There exist $u^+\in SX^+$ and $R>r>0$ such that
\begin{equation}\label{eq:LinkingGeometry}
\sup_{\pa M(u^+)} J\leq 0=J(0)<\inf_{S^+_r}J
\end{equation}
where
\begin{eqnarray*}
\pa M(u^+)&:=&\{u=t u^++\tu \in X: \tu \in \tX,\ (\|u\|=R,t\geq 0)\textnormal{ or }(\|u\|\leq R,t=0) \},\\
S^+_r&:=&\{ u^+\in X^+,\|u^+\|=r\}.
\end{eqnarray*}
However we are not able to apply any linking result or Nehari-Pankov manifold technique to find critical points of $J$. Moreover, even if (S) holds and $\la<0$ we do not know whether there is a least energy solution.\\

{\bf Problem 3.} Solutions of \eqref{eq:main} lead to semi-trivial solutions $(E,0)$, $(0,E)$ or diagonal solutions $(E,E)$ of the system \eqref{eq:system}. Are there other solutions of the system? For $V=\lambda\id_{3\times 3}$ with $\la>0$ there are families of solutions $(E_\la,0)$, $(0,E_\la)$, $(E_\la,E_\la)$. Is there bifurcation from these solutions as in \cite{Bartsch-Dancer-Wang:2010}? \\

{\bf Problem 4.} In this survey only subcritical nonlinearities have been treated. We are only aware of the two preprints \cite{MederskiENZ, Zeng:2016} dealing with critical nonlinearities. In order to deal with critical nonlinearities it would be very useful to find ``soliton" solutions $u:\R^3\to\R^3$ of $\curlop\curlop u=|u|^4u$. We conjecture that these have cylindrical symmetry and are of the form \eqref{eq:sym1}.\\

{\bf Problem 5.} For natural materials the permittivity $\eps$ is positive and bounded away from $0$, hence $\essinf_{x\in\Om} V>0$. If $\Om=\R^3$ there are so far no results dealing with this case. The problem is very difficult because then $0$ lies in the essential spectrum of the operator $\curlop\curlop-V(x)$, even on the space of divergence-free fields.

{\bf Problem 6.} Another very challenging problem is to consider the wave equation \eqref{eq:wave}
\[
  \eps(x)\pa_t^2\cE + \pa_t^2\cP_{NL}(x,\cE) + \curlop (\mu(x)^{-1}\curlop \cE) = 0
\]
without assuming that the scalar susceptibility $\chi$ depends only on the intensity $|E_1|^2+|E_2|^2$ of the time-harmonic solution
\[
  \cE(x,t) = E_1(x)\cos(\om t) + E_2(x)\sin(\om t)\quad\text{for $x\in\Om$ and $t\in\R$.}
\]
The nonlinear polarization is of the form $\cP(x,\cE)=\chi(\cE)\cE$ so \eqref{eq:wave} does not lead to an elliptic equation for $E_1,E_2$. It is completely unclear how to obtain solutions of \eqref{eq:wave}, say for $\cP(x,\cE)=|\cE|^2\cE$.

{\sc Address of the authors:}\\[1em]
\parbox{8cm}{
	\vspace{-27mm}
	Thomas Bartsch\\
 Mathematisches Institut\\
 Universit\"at Giessen\\
 Arndtstr.\ 2\\
 35392 Giessen, Germany\\
 Thomas.Bartsch@math.uni-giessen.de}
\parbox{10cm}{
\vspace{2mm}
 Jaros\l aw Mederski\\
 Institute of Mathematics,
 \newline\indent
 Polish Academy of Sciences
 \newline\indent
 ul. \'Sniadeckich 8, 00-956
 Warszawa, Poland\\
 and\\
 Nicolaus Copernicus University, \\
 Faculty of Mathematics and Computer Science\\
 ul.\ Chopina 12/18\\
 87-100 Toru\'n, Poland\\
 jmederski@impan.pl\\
 }


\begin{thebibliography}{99}
\baselineskip 2 mm

\bibitem{Ambrosetti-Rabinowitz:1973} A. Ambrosetti, P.H. Rabinowitz: {\em Dual variational methods in critical point theory and applications}, J. Funct. Anal. {\bf 14} (1973), 349--381.

\bibitem{Amrouche}  C. Amrouche, C. Bernardi, M. Dauge, V. Girault: {\em Vector potentials in three-dimensional non-smooth domains}, Math. Methods Appl. Sci. {\bf 21} (1998), no. 9, 823--864.

\bibitem{BenForAzzAprile} A. Azzollini, V. Benci, T. D'Aprile, D. Fortunato: {\em Existence of Static Solutions of the Semilinear Maxwell Equations}, Ric. Mat. {\bf 55} (2006), no. 2, 283--297.

\bibitem{BadialePisaniRolando} M. Badiale, L. Pisani, S. Rolando: {\em Sum of weighted Lebesgue spaces and nonlinear elliptic equations}, Nonlinear Differ. Equ. Appl. {\bf 18} (2011), 369--405.

\bibitem{Ball2012} J. M. Ball, Y. Capdeboscq, B. Tsering-Xiao: {\em On uniqueness for time harmonic anisotropic Maxwell's equations with piecewise regular coefficients}, Math. Models Methods Appl. Sci. {\bf 22} (2012), no. 11, 1250036, 11 pp.

\bibitem{Bartolo-etal:1983} P. Bartolo, V. Benci, D. Fortunato: {\em Abstract critical point theorems and applications to some nonlinear problems with strong resonance at infinity}, Nonlin. Anal. Theory Meth. Appl. {\bf 7} (1983), 981--1012.

\bibitem{Bartsch:1993} T. Bartsch: {\em Infinitely many solutions of a symmetric Dirichlet problem}, Nonlinear Analysis, Theory, Meth. Appl. {\bf 20} (1993), no. 12, 1205--1216.

\bibitem{Bartsch-Dancer-Wang:2010} T. Bartsch, N. Dancer, Z.-Q. Wang: {\em  A Liouville theorem, a-priori bounds, and bifurcating branches of positive solutions for a nonlinear elliptic system}, Calc. Var. Partial Diff. Equ. {\bf 37} (2010), 345--361.

\bibitem{BartschDing} T. Bartsch, Y. Ding: {\em Deformation theorems on non-metrizable vector spaces and applications to critical point theory}, Mathematische Nachrichten {\bf 279} (2006), no. 12, 1267--1288.

\bibitem{BDPR:2016} T. Bartsch, T. Dohnal, M. Plum, W. Reichel: {\em Ground States of a Nonlinear Curl-Curl Problem in Cylindrically Symmetric Media}, Nonlin. Diff. Equ. Appl. DOI 10.1007/s00030-016-0403-0, arXiv:1411.7153.


\bibitem{BartschMederski1} T. Bartsch, J. Mederski: {\em Ground and bound state solutions of semilinear time-harmonic Maxwell equations in a bounded domain}, Arch. Rational Mech. Anal. {\bf 215}(1), (2015), 283--306.

\bibitem{BartschMederski2} T. Bartsch, J. Mederski: {\em Nonlinear time-harmonic Maxwell equations in an anisotropic bounded domain}, arXiv:1509.01994.

\bibitem{Bauer-Pauly-Schomburg:2016} S. Bauer, D. Pauly, M. Schomburg: {\em The Maxwell compactness property in bounded weak Lipschitz domains with mixed boundary conditions}, SIAM J. Math. Anal. {\bf 48}(4) (2016), 2912--2943.

\bibitem{BenFor} V. Benci, D. Fortunato: {\em Towards a unified field theory for classical electrodynamics}, Arch. Rational Mech. Anal. {\bf 173} (2004), 379--414.

\bibitem{BenciRabinowitz} V. Benci, P. H. Rabinowitz: {\em Critical point theorems for indefinite functionals}, Invent. Math. {\bf 52} (1979), no. 3, 241--273.

\bibitem{BuffaAmmariNed} A. Buffa, H. Ammari, J. C. Nédélec: {\em Justification of Eddy Currents Model for the Maxwell Equations}, SIAM J. Appl. Math. {\bf 60} (5), (2000), 1805--1823.

\bibitem{Clark:1972} D. C. Clark: {\em A variant of Ljusternik-Schnirelmann theory}, Indiana Univ. Math. J. {\bf 22} (1972), 65--74.


\bibitem{CostabelDN1999} M. Costabel, M. Dauge, S. Nicaise: {\em Singularities of Maxwell interface problems}, Math. Model. Numer. Anal. {bf 33}, (1999), 627--649.

\bibitem{Costabel} M. Costabel: {\em A remark on the regularity of solutions of Maxwell's equations on Lipschitz domains}, Math. Methods Appl. Sci. {\bf 12}, (1990), 365--368.

\bibitem{CotiZelatiRab} V. Coti Zelati, P.H. Rabinowitz: {\em Homoclinic type solutions for a semilinear elliptic PDE on $\R^N$}, Comm. Pure and Applied Math. {\bf 45}, no. 10, (1992), 1217--1269.

\bibitem{DAprileSiciliano} T. D'Aprile, G. Siciliano: {\em Magnetostatic solutions for a semilinear perturbation of the Maxwell equations}, Adv. Differential Equations {\bf 16} (2011), no. 5--6, 435-466.

\bibitem{DingBook} Y. Ding: {\em Variational Methods for Strongly Indefinite Problems}, Interdisciplinary Mathematical Sciences {\bf 7}, World Scientific Publishing 2007.

\bibitem{Doerfler} W. D\"orfler, A. Lechleiter, M. Plum, G. Schneider, C. Wieners: {\em Photonic Crystals: Mathematical Analysis and Numerical Approximation}, Springer Basel 2012.

\bibitem{GilbargTrudinger} D. Gilbarg, N.S. Trudinger: {\em Elliptic partial differential equations of second order}, Springer-Verlag, Berlin, 2001.

\bibitem{Hiptmair} R. Hiptmair: {\em Finite elements in computational electromagnetism}, Acta Numerica {\bf 11}, (2002), 237--339.

\bibitem{HirschReichel} A. Hirsch, W. Reichel: {\em Existence of cylindrically symmetric ground states to a nonlinear curl-curl equation with non-constant coefficients}, arXiv:1606.04415.

\bibitem{KirschHettlich} A. Kirsch, F. Hettlich: {\em The Mathematical Theory of Time-Harmonic Maxwell's Equations: Expansion-, Integral-, and Variational Methods}, Springer 2015.


\bibitem{Leis68} R. Leis: {\em Zur Theorie elektromagnetischer Schwingungen in anisotropen inhomogenen Medien}, Math. Z. {\bf 106 } (1968), 213--224.

\bibitem{MederskiENZ} J. Mederski: {\em Ground states of time-harmonic semilinear Maxwell equations in $\R^3$ with vanishing permittivity}, Arch. Rational Mech. Anal. 218 (2), (2015), 825--861.

\bibitem{MederskiSurv} J. Mederski: {\em Nonlinear time-harmonic Maxwell equations in $\R^3$: recent results and open questions}, Lecture Notes of {\em Seminario Interdisciplinare di Matematica} Vol. {\bf 13} (2016), 47--57.

\bibitem{MederskiSystem} J. Mederski: {\em Ground states of a system of nonlinear Schr\"odinger equations with periodic potentials}, Comm. Partial Differential Equations {\bf 41} (9), (2016), 1426--1440.

\bibitem{MederskiMaxwellCritical} J. Mederski: {\em The Brezis-Nirenberg problem for the curl-curl operator}, submitted arXiv:1609.03989.

\bibitem{Monk} P. Monk: {\em Finite Element Methods for Maxwell's Equations}, Oxford University Press 2003.

\bibitem{Nie} W. Nie: {\em Optical Nonlinearity: Phenomena, applications, and materials}, Adv. Mater. {\bf 5}, (1993), 520--545.

\bibitem{Okaji:2002} T. Okaji: {\em Strong unique continuation property for time harmonic Maxwell equations}, J. Math. Soc. Japan {\bf 54} (2002), 89--122.

\bibitem{Pankov} A. Pankov: {\em Periodic Nonlinear Schr\"odinger Equation with Application to Photonic Crystals}, Milan J. Math. {\bf 73} (2005), 259--287.

\bibitem{Picard01} R. Picard, N. Weck, K.-J. Witsch: {\em Time-harmonic Maxwell equations in the exterior of perfectly conducting, irregular obstacles}, Analysis (Munich) 21 (2001), no. 3, 231--263.


\bibitem{Qin-Tang:2016} D. Qin, X. Tang: {\em Time-harmonic Maxwell equations with asymptotically linear polarization}, Z. Angew. Math. Phys. {\bf 67} (2016), no. 3, 719–740.

\bibitem{Rabinowitz:1986} P. Rabinowitz: {\em Minimax Methods in Critical Point Theory with Applications to Differential Equations}, CBMS Regional Conference Series in Mathematics, Vol. {\bf 65}, Amer. Math. Soc., Providence, Rhode Island 1986.

\bibitem{Pistoia-Ramos:2008} A. Pistoia, M. Ramos: {\em Locating the peaks of the least energy solutions to an elliptic system with Dirichlet boundary conditions}, NoDEA Nonlinear Differential Equations Appl. {\bf 15} (2008), no. 1, 1-–23.

\bibitem{FundPhotonics} B.E.A. Saleh, M.C. Teich: {\em Fundamentals of Photonics}, 2nd Edition, Wiley 2007.


\bibitem{StruweSplitting} M. Struwe: {\em A global compactness result for elliptic boundary value problems involving limiting nonlinearities}, Math. Z. {\bf 187} (1984), no. 4, 511--517.

\bibitem{Stuart:1991} C. A. Stuart: {\em Self-trapping of an electromagnetic field and bifurcation from the essential spectrum}, Arch. Rational Mech. Anal. {\bf 113} (1991), no. 1, 65--96.

\bibitem{Stuart:1993} C. A. Stuart: {\em Guidance Properties of Nonlinear Planar Waveguides}, Arch. Rational Mech. Anal. {\bf 125} (1993), no. 1, 145--200.


\bibitem{StuartZhou96} C.A. Stuart, H.S. Zhou: {\em A variational problem related to self-trapping of an electromagnetic field}, Math. Methods Appl. Sci. {\bf 19} (1996), no. 17, 1397--1407.


\bibitem{StuartZhou03} C.A. Stuart, H.S. Zhou: {\em A constrained minimization problem and its application to guided cylindrical TM-modes in an anisotropic self-focusing dielectric}, Calc. Var. Partial Differential Equations {\bf 16} (2003), no. 4, 335--373.

\bibitem{StuartZhou05} C.A. Stuart, H.S. Zhou: {\em Axisymmetric TE-modes in a self-focusing dielectric}, SIAM J. Math. Anal. {\bf 37} (2005), no. 1, 218--237.

\bibitem{StuartZhou10} C.A. Stuart, H.S. Zhou: {\em Existence of guided cylindrical TM-modes in an inhomogeneous self-focusing dielectric}, Math. Models Methods Appl. Sci. {\bf 20} (2010), no. 9, 1681--1719.

\bibitem{SzulkinWethHandbook} A. Szulkin, T. Weth: {\em The method of Nehari manifold}, Handbook of nonconvex analysis and applications, 597--632, Int. Press, Somerville, 2010.


\bibitem{Tang-Qin:2016} X. Tang, D. Qin: {\em Ground state solutions for semilinear time-harmonic Maxwell equations}, J. Math. Phys. {\bf 57} (2016), no. 4, 041505.

\bibitem{Vogelsang:1991} V. Vogelsang: {\em On the strong unique continuation principle for inequalities of Maxwell type}, Math. Ann. {\bf 289} (1991), 285--295.


\bibitem{Zeng:2016} X. Zeng: {\em Cylindrically symmetric ground state solutions for curl-curl equations with critical exponent}, arXiv:1609.09598.

\end{thebibliography}
\end{document}